\documentclass[10pt,reqno]{amsart}
\usepackage{amsmath,amsthm,mathrsfs,amsfonts,amssymb,amscd,euscript,graphicx,bm,color}
\usepackage[all]{xy}

\numberwithin{equation}{section}

\DeclareMathOperator{\supp}{supp}
\DeclareMathOperator{\Dist}{Dist}

\DeclareMathOperator{\Leb}{Leb}

\newcommand{\real}{{\rm I\!R}} 

\newcommand{\natur}{{\rm I\!N}} 

\newcommand{\ine}{{\,\rm int\,}} 

\newcommand{\am}{\left |} 

\newcommand{\fm}{\right |} 

\newcommand{\ak}{\left\{} 

\newcommand{\fk}{\right\}} 

\newcommand{\ac}{\left [} 

\newcommand{\fc}{\right ]} 

\newtheorem{mtheorem}{Theorem}
\newtheorem{theorem}{Theorem}[section]
\newtheorem*{theorem*}{Theorem}
\newtheorem{lemma}[theorem]{Lemma}

\newtheorem{corollary}[theorem]{Corollary}
\newtheorem{proposition}[theorem]{Proposition}
\newtheorem{claim}{Claim}

\theoremstyle{definition}
\newtheorem{definition}[theorem]{Definition}
\newtheorem{example}{Example}
\newtheorem*{example*}{Example}
\newtheorem{remark}[theorem]{Remark}
\newtheorem*{remark*}{Remark}

\begin{document}

\title[Invariant multi-graphs]{Invariant multi-graphs in step skew-products}

\author[K. Gelfert]{Katrin Gelfert}
\address{Instituto de Matem\'atica, Universidade Federal do Rio de Janeiro, Cidade Universit\'aria - Ilha do Fund\~ao, Rio de Janeiro 21945-909,  Brazil}
\email{gelfert@im.ufrj.br}
\author[D. Oliveira]{Daniel Oliveira}\address{Instituto de Ci\^encias Exatas, Universidade Federal Rural do Rio de Janeiro, Rodovia BR-465, Km 7, Serop\'edica, RJ, Brazil}\email{danielreis@ufrrj.br}

\begin{abstract}
We study step skew-products over a finite-state shift (base) space whose fiber maps are  $C^1$ injective maps on the unit interval.  We show that certain invariant sets have a multi-graph structure and can be written graphs of one, two or more functions defined on the base. In particular, this applies to any hyperbolic set and to the support of any  ergodic hyperbolic measure.
Moreover, within the class of step skew-products whose interval maps are ``absorbing'', open and densely the phase space decomposes into attracting and  repelling double-strips such that their attractors and repellers are graphs of one single-valued or bi-valued continuous function almost everywhere, respectively.
\end{abstract}

\begin{thanks}{This research has been supported [in part] by CNPq- and FAPERJ-grants (Brazil).}\end{thanks}

\maketitle

\section{Introduction}

We study the following class of maps.

\begin{definition}\label{defdeF}
Given $N \in \natur$, denote by $\Sigma_N = \ak 1,\ldots,N \fk^{\mathbb{Z}}$ the space of bilateral sequences of $N$ symbols and $I = [0,1]$ the unit interval. We consider a \emph{step skew-product} defined by
\begin{equation}\label{eq:stepskewproood}
    F\colon \Sigma_N \times I \to \Sigma_N \times I \quad \colon \quad (\xi,p) \mapsto (\sigma(\xi), f_{\xi_0}(p)),
\end{equation}
where $\sigma\colon \Sigma_N \to \Sigma_N$ is the usual shift map and  $f_i\colon I \to f_i(I) \subset I$, $i=1,\ldots,N$.
We  always assume that $f_i$ are $C^1$-diffeomorphisms onto its image. We will call $\Sigma_N$ the \emph{base} and $I$ (and sometimes also the sets $I_\xi:=\{\xi\}\times I$, $\xi\in\Sigma_N$) the \emph{fiber} and refer to $f_i$ as \emph{fiber maps}.
\end{definition}

This class of maps recently has drawn certain attention as they provide important examples of partially hyperbolic dynamical systems (see, for example, \cite{DG,DHRS}). More precisely, those partially hyperbolic systems have a quite specific structure because the central direction (associated to the fiber direction) is one-dimensional and integrable. Nevertheless, they still include maps with a very complex dynamics as the invariant central direction can mix hyperbolicity of  different types (can have negative or positive fiberwise Lyapunov exponents) or can be non-hyperbolic and as such can provide important toy models of genuinely non-hyperbolic systems.
Step skew-products have the advantage that they admit a very simple description and their analysis is technically relatively easy. Because of that, they provide a good structure to investigate non-hyperbolicity and hopefully gradually study more difficult problems (for example, to  pass from step skew-products to general skew-products and eventually to certain partially hyperbolic maps).
The analysis of step skew-products is also intimately related to the one of iterated function systems, which have an intrinsic interest.

It is well known that for skew-products with monotone fiber maps there is a close relation between (the existence of) invariant graphs and ergodic measures of $F$ (see, for example,  \cite[Section 1.8]{Ar} for results on general random dynamical systems which can be one point of view to analyse \eqref{eq:stepskewproood}).
There exist also already some specific results on the regularity of invariant graphs  (see, for example, \cite{HNW,St,St1}), however in general assuming that the fiber dynamics is hyperbolic.

Further results  that give a good description of skews products from an ergodic point of view are due to Ruelle and Wilkinson \cite{RW} who consider a general measurable skew-product with an invertible and ergodic map in the base (which is a general probability space) and $C^{1+\alpha}$ diffeomorphisms on the fiber (which is a general Riemannian compact manifold). They prove that any ergodic measure for the skew-product which projects to the ergodic measure in the base and which has only negative fiberwise Lyapunov exponents admits an atomic fiber disintegration.
Results about atomicity of the measures in the disintegration were also obtained in other contexts (see \cite{HP}, for example).

We will provide a detailed description, both from a topological and from an ergodic point of view, of the step skew-product \eqref{eq:stepskewproood}. More precisely, we will establish the existence of attractors and repellers and study further  relevant sets such as invariant (multi-/bi-)graphs.
Our results essentially split into two parts. First, we study (multi-/bi-)graphs and some mild hyperbolicity assumption (Theorems \ref{prop:conjhip} and \ref{prop:medhip} below). Second, fixing  some Markov measure on the base, we investigate a generic class of step skew-products \eqref{eq:stepskewproood} assuming also that $I$ is absorbing and deduce a complete description of the topological and ergodic structures.
One of our starting points for the latter is the work from Klepstyn and Volk \cite{KV}, where they restrict themselves to orientation preserving fiber maps only.
One novelty of this paper is on one hand precisely that we do not \emph{a priori} assume uniform hyperbolicity of the fiber dynamics. Moreover, comparing with \cite{RW}, the fact that the fibers are one-dimensional, allows us to require $C^1$ regularity only and enables us to improve the conclusion about the atomic disintegration and to show the existence of a (bi-)graph structure. On the other hand that we do not assume that fiber maps are orientation preserving and hence, in general, do not have invariant ``simple"  graphs.  The hypothesis that every fiber map preserves orientation appears quite often in the literature (see, for example, \cite{GH,I,KV}), but restricts a lot the class of covered examples.

Let us provide some more details on  \cite{KV} and discuss further results. In fact, the authors study  \eqref{eq:stepskewproood} assuming additionally that all maps preserve orientation and that $I$ is absorbing and show the existence of a finite collection of trapping and repelling strips whose union is the entire phase space such that every trapping (repelling) strip has a unique maximal attractor (repeller). These attractors and repellers are \emph{bony graphs} (following the notation from \cite{KV}, see also Definition \ref{def:bigraficoossudo}), that is, each of them intersects almost every fiber (with respect to an \emph{a priori} fixed base Markov measure) in an unique point and the other fibers in an interval. In a certain way, these sets are similar to the porcupine-like horseshoes studied in \cite{DG}, a paradigmatic example of partially hyperbolic dynamical system (although one of the characteristic features of this example is the presence of an orientation reversing map giving rise to a cycle).
Building also on  \cite{KV}, the particular case where all fiber maps fix the extremes of $I$ (and hence trivially also preserve orientation) is studied in \cite{GH} with certain additional assumptions on the fiberwise Lyapunov exponents on the extremes. The authors obtain a similar  decomposition of the phase space as in \cite{KV}. However, in the case of the border strips (following the notation from \cite{GH}), the maximal attractor is a \emph{thick graph} (see also \cite{I}), that is, it is a graph whose closure has full measure (where the measure considered is the product of the Markov measure on the base with the Lebesgue measure on the fiber).

Let us now state more precisely our results.
Consider the \emph{maximal invariant  set of $F$} defined by
\begin{equation}\label{defdeLambda}
    \Lambda := \bigcap_{n\geq 0} F^n (\Sigma_N \times I).
\end{equation}
 From now on we will always consider $F|_\Lambda$, but for simplicity we will only write $F$.
 Note that $F^{-1}$ is well defined on $\Lambda$.
We will denote by $\Pi_1\colon\Sigma_N\times I\to\Sigma_N$ and $\Pi_2\colon\Sigma_N\times I\to I$ the canonical projections onto the first and second coordinates, respectively.
Given $H \subset \Lambda$, let
\begin{equation}\label{eq:projfib}
    H_\xi := \Pi_2(H \cap (\{\xi\}\times I))\subset I.
\end{equation}
We will denote $f_{\xi_{k}\ldots\xi_{m}} := f_{\xi_{m}} \circ \ldots \circ f_{\xi_{k}}$ for each $k,m \in \mathbb{N}, k \leq m$.

Let us recall the concept of hyperbolicity in a step skew-product.

\begin{definition}\label{def:hip}
A set $H\subset\Lambda$ is \emph{hyperbolic with fiber contraction} (with respect to $F$) if there exist $c > 0$ and $0 < \lambda < 1$ such that
\begin{equation}\label{eq:hipcont}
    \am (f_{\xi_0  \ldots  \xi_{n-1}})'(p) \fm \leq c \lambda^n \quad\text{for all}\,\, n \geq 1 \,\,\text{and for all}\,\, (\xi,p) \in H.
\end{equation}
A set $H$ is \emph{hyperbolic with fiber expansion} (with respect to $F$) if it is hyperbolic with fiber contraction (with respect to $F^{-1}$).
We say that $H$ is \emph{hyperbolic} if it is hyperbolic with either fiber contraction or fiber expansion.
\end{definition}

The following is the first main result.

\begin{mtheorem}\label{prop:conjhip}
Let $F$ be as Definition \ref{defdeF} and $\Lambda$ its maximal invariant set.
Then for every $F$-invariant hyperbolic set $H \subset \Lambda$ there exists $M \geq 1$ such that $\# H_\xi \leq M$ for all $\xi \in \Sigma_N$.
\end{mtheorem}

The second main result concerns the more general case of a set $H$ which supports a hyperbolic measure. Given an ergodic $F$-invariant Borel probability measure $\mu$, define the \emph{(fiber) Lyapunov exponent} of $\mu$ by
\begin{equation}\label{eq:expdemu}
	\chi(\mu)
	:= \int\log\,\lvert(f_{\xi_0})'(p)\rvert\,d\mu(\xi,p).
\end{equation}
We will call $\mu$ \emph{hyperbolic} if $\chi(\mu)\ne0$.
By ergodicity, for $\mu$-almost every point $(\xi,p)$ we have $\chi(\mu)=\chi(\xi,p)$, where $\chi(\xi,p)$ denotes the \emph{(forward) Lyapunov exponent} of $(\xi,p)$ (with respect to $F$), see Section~\ref{sec:chmh}.

\begin{mtheorem}\label{prop:medhip}
	Let $F$ be as Definition \ref{defdeF} and $\Lambda$ its maximal invariant set.  Let $\mu$ be a hyperbolic ergodic $F$-invariant probability measure and $H\subset \Lambda$ an $F$-invariant set such that $\mu(H)=1$ and $\chi(\xi,p)=\chi(\mu)$ for every $(\xi,p)\in H$. Then there exists $M \geq 1$ such that $\# H_\xi \leq M$ for every $\xi \in \Pi_1(H)$. Moreover, there exist functions $\varphi^-, \varphi^+ \colon\Pi_1(H) \to I$ such that
\[
    \Gamma (\mu) :=  \Gamma^- \cup \Gamma^+,\quad
    \text{ where }\quad
    \Gamma^\pm := \ak (\xi,\varphi^\pm(\xi)) \colon \xi \in \Pi_1(H) \fk
\]
is an $F$-invariant set which coincides with the support of $\mu$ in $\mu$-almost every point.
\end{mtheorem}

Finally, we will study a certain subclass of step skew-products and obtain a more detailed information about its graph-structure (topologically and ergodically). Let us first provide some definitions.

\begin{definition}\label{def:multifuncao}
A \emph{multi-function} $\psi\colon D \subset \Sigma_N \to I$ is a relation that associates to every point $\xi \in D$ a nonempty subset $\psi(\xi) \subset I$. A multi-function $\psi\colon D \to I$ is \emph{compact-valued} when $\psi(D)$ is  compact for every $\xi \in D$ and it is \emph{uniformly finite} if there exists $M \geq 1$ such that $\# \psi(\xi) \leq M$ for all $\xi \in D$. If $\#\psi(\xi)=2$ for all $\xi \in D$ then $\psi$ is a \emph{bi-function}.

Given a multi-function $\psi$, we call the set $\ak (\xi,\psi(\xi)) \colon \xi \in D \fk$ a \emph{multi-graph} in $\Sigma_N \times I$. Analogously for bi-function and \emph{bi-graph}.
\end{definition}
Following~\cite{KV} we introduce the following (slightly extended) concept.

\begin{definition}[Continuous (Bi-)bony graph]\label{def:bigraficoossudo}
Consider a Borel probability measure $\lambda$ on $\Sigma_N$. A set $B \subset \Sigma \times I$ is a \emph{bony graph} (relative to $\lambda$) if  $B_\xi$ contains a single point for $\lambda$-almost every $\xi$ and is an interval in all remaining points.
A bony graph is \emph{continuous} if for all $\xi \in \Pi_1(B)$ and for all $\varepsilon > 0$ there exists $\delta > 0$ such that if
$\eta \in \Pi_1(B)$ and $d(\eta,\xi) < \delta$ then $B_\eta \subset U_\varepsilon(B_\xi)$, where $U_\varepsilon(B_\xi)$ denotes the $\varepsilon$-neighborhood of $B_\xi$.
A set $B\subset\Sigma\times I$ is a \emph{(continuous) bi-bony graph} (relative to $\lambda$) if it is a union of two (continuous) bony graphs (relative to $\lambda$).
\end{definition}

The interval $I$ is \emph{absorbing} with respect to a fiber map $f_i$ if $f_i(I)\subset\ine(I)$.

The following is our third main result. We postpone the definition of \emph{attracting} and \emph{repelling bi-strips} to Section~\ref{subsec:far} and the details on the topology in the space of step skew-products $\mathcal{S}(N)$ to Section \ref{subsec:cg}.

\begin{mtheorem}\label{teo:TP}
	Let $F$ be as Definition \ref{defdeF}. Assume that  $I$ is absorbing with respect to each fiber maps $f_i$, $i=1,\ldots,N$.
	Then there exists an open and dense subset of $\mathcal S'\subset\mathcal{S}(N)$ such that for every $F\in\mathcal S'$ there exists a finite collection of attracting and repelling bi-strips which satisfy the following property:
\begin{enumerate}
    \item[(1)] Their union is the phase space $\Sigma_N\times I$.
    \item[(2)] Each attracting bi-strip has a unique maximal attractor and each repelling bi-strip has a unique maximal repeller.
\end{enumerate}
	Moreover, if $\lambda_0$ is a (nondegenerate) Markov measure on $\Sigma_N$ then:
\begin{enumerate}
    \item[(3)] The maximal attractor (repeller) $A_i \subset S_i$ in each attracting (repelling) bi-strip is a continuous bi-bony graph.
    \item[(4)] Each attracting (repelling) bi-strip $S_i$ supports a unique $F$-invariant and ergodic measure $\mu_i$ which projects to $\lambda_0$ and whose fiberwise disintegration is atomic in $\lambda_0$-almost every fiber.
    Moreover, this measure is physical and hyperbolic and its basin contains a full measure set relative to the measure $(\lambda_0\times m)|_{S_i}$.
    Further, there exists a subset $\Upsilon\subset A_i$ such that $\mu_i(\Upsilon)=1$ and $\Pi_1 \colon \Upsilon \to \Pi_1(\Upsilon)$ is a semi-conjugation two-to-one.
\end{enumerate}
\end{mtheorem}

As already mentioned above, Theorem \ref{teo:TP} generalizes~\cite[Theorem 2.15]{KV} in the sense that we remove the hypothesis that every fiber map preserves orientation. Note that in the case that all maps do preserve orientation, then each bi-strip and bi-bony graph is simply a strip and a bony graph, respectively, and the semi-conjugation claimed in item 4 is a conjugation. In the more general case, new phenomena occur which are precisely the focus of this paper.

Section \ref{sec:chmh} studies general multi-graphs and contains the proofs of Theorems \ref{prop:conjhip} and \ref{prop:medhip}.
Section \ref{sec:dat} investigates the properties of generic step skew-products claimed in Theorem \ref{teo:TP}. Our main approach is to introduce an extended step skew-product (Section \ref{subsec:ESSP}) with orientation preserving fiber maps and studies correspondingly defined strips (Section \ref{subsec:far}), extended Markov measures (Section \ref{subsec:me}) and stationary measures (Section \ref{subsec:sm}). The proof of Theorem \ref{teo:TP} is  concluded in Section \ref{subsec:demTP}.

\section{Multi-graphs in step skew-products}\label{sec:chmh}

In this section we will prove Theorems \ref{prop:conjhip} and \ref{prop:medhip}. We will also see some examples for these two results.

\subsection{Proof of Theorem \ref{prop:conjhip}}

We will  assume that $H$ is hyperbolic with fiber contraction, the other case is analogous and will be omitted.

To sketch the proof, we will show that there are uniform neighborhoods (in the fibers) of each point in $H$ being uniformly contracted by some iterate of $F$ and hence uniformly expanded by the inverse  $F^{-1}$. By compactness of the fiber,   simultaneously there can  exist only a bounded number of them.

Let $c$ and $\lambda$ be as in \eqref{eq:hipcont}. Let $n_0$ be the smallest positive integer such that $\lambda_1:=c\lambda^{n_0} < 1$. Hence
\[
    \am (f_{\xi_0\ldots\xi_{nn_0-1}})'(p) \fm \leq \lambda_1^n \quad \text{for all}\,\,n \geq 1 \,\,\text{and for all}\,\,(\xi,p) \in H.
\]

By contradiction, suppose that for each $M \geq 1$ there exists $\xi \in \Sigma_N$ so that $\# H_\xi > M$.

Fix  $\lambda_2 \in (\lambda_1,1)$. By compactness of $I$ there exists $\varepsilon > 0$ so that
\begin{equation}\label{eq:derivadaforadeH}
    \am (f_{\xi_0\ldots\xi_{n_0-1}})'(x) \fm \leq \lambda_2 \quad \text{for all}\,\,x \in (p-\varepsilon, p +\varepsilon) \cap I\,\,\text{and for all}\,\,(\xi,p) \in H.
\end{equation}
Fix $N_1 \in \natur$ and $M_0 \in \natur$ satisfying\footnote{$\left\lceil a \right\rceil:= \max \ak k \in \mathbb{Z} \colon k \leq a \fk$}
\begin{equation}\label{eq:escdeM_0}
    N_1 \varepsilon >1
    \quad\text{ and }\quad
     \left\lceil\frac{M_0}{N_1^{N_1}} \right\rceil > 1.
\end{equation}
Let $\xi =(\ldots\xi_{-1}\xi_0\xi_1\ldots)\in \Sigma_N$ such that $\# H_\xi > M_0$. Let $g=f_{\xi_{-n_0}\ldots\xi_{-1}}$. By $F$-invariance of $H$ and the fact that the fiber maps are local diffeomorphisms, we have $\# H_{\sigma^k(\xi)} > M_0$ for all $k \in \mathbb{Z}$.
Let $q_1 < \cdots < q_{M_0}$ be an enumeration of $M_0$ points in $H_\xi$. By monotonicity of the fiber maps, we know that $g^{-1}(q_i)$ is between $g^{-1}(q_1)$ and $g^{-1}(q_{M_0})$ for all $i=2,\ldots,M_0-1$.  Fix $n_1 \in \natur$ such that
\begin{equation}\label{contradicao1}
    \Big(\frac{1}{\lambda_2}\Big)^{n_1} \am q_{M_0} - q_1 \fm > 1.
\end{equation}
We consider two cases.

\smallskip\noindent\textbf{Case 1:} If any interval of length $\varepsilon$ between $g^{-1}(q_1)$ and $g^{-1}(q_{M_0})$ intersects $H_{\sigma^{-n_0}(\xi)}$, then 
 \eqref{eq:derivadaforadeH} implies
\[
    \am g'(x) \fm \leq \lambda_2 \, \text{for all}\,\,x \,\,\text{between}\,\,g^{-1}(q_1)\,\,\text{and}\,\, g^{-1}(q_{M_0}),
\]
        and hence, by the Mean Value Theorem, we have that
\begin{equation}\label{eq:expansao}
  	\am q_1 - q_{M_0} \fm  \leq \lambda_2 \am g^{-1}(q_1) - g^{-1}(q_{M_0}) \fm.
\end{equation}

\smallskip\noindent\textbf{Case 2:} Otherwise, denote by $J_1,\ldots,J_l$, $l \geq 1$, the maximal (disjoint) intervals  of length at least $\varepsilon$ whose intersections with $H_{\sigma^{-n_0}(\xi)}$ are empty.
By \eqref{eq:escdeM_0}, we have
\begin{equation}\label{eq:cotaparal}
    l < N_1.
\end{equation}
Assume that $J_i=(a_i,b_i)$ with $b_i=a_{i+1}$ for $i=1,\ldots,l-1$.
Those intervals are sent by the dynamics of $F^{-n_0}$ into intervals again of length at least $\varepsilon$ (and also not intersecting $H$).
Indeed, as the fiber maps are injective, $g^{-1}(q_i)$, $i=1,\ldots,M_0$, are pairwise distinct points in $H_{\sigma^{-n_0}(\xi)}$ and by the Pigeonhole Principle and \eqref{eq:cotaparal}, at least $M_1 := \left\lceil M_0/N_1 \right\rceil$ of them,  say $q_1^1 < \ldots < q_{M_1}^1$,  are either between $0$ and $J_1$ or  between $J_{i_0}$ and $J_{i_0+1}$ for some $i_0=1,\ldots,l-1$ or  between $J_l$ and $1$. Note that $M_1 \geq 2$ by \eqref{eq:escdeM_0}. Without loss of generality, assume that $q_1^1,\ldots, q_{M_1}^1$ are between $J_{i_0}$ e $J_{i_0+1}$. The other cases are analogous.
For each $q \leq a_{i_0}$ such that $q \in H_{\sigma^{-n_0}(\xi)}$, by \eqref{eq:derivadaforadeH} and the Mean Value Theorem,  we have
\[
    \am (f_{\xi_{-2n_0}\ldots\xi_{-n_0-1}})^{-1}(q) - (f_{\xi_{-2n_0}\ldots\xi_{-n_0-1}})^{-1}\Big(q +\frac{\varepsilon}{2}\Big) \fm
    \geq \frac{1}{\lambda_2} \frac{\varepsilon}{2}
    > \frac{\varepsilon}{2},
\]
and similarily
\[
    \am (f_{\xi_{-2n_0}\ldots\xi_{-n_0-1}})^{-1}(q_1^1) - (f_{\xi_{-2n_0}\ldots\xi_{-n_0-1}})^{-1}\Big(q_1^1 -\frac{\varepsilon}{2}\Big) \fm  \geq \frac{1}{\lambda_2} \frac{\varepsilon}{2} > \frac{\varepsilon}{2}.
\]
As  $q_1^1 -q\geq \varepsilon$, monotonicity of the fiber maps give
\[
    \am (f_{\xi_{-2n_0}\ldots\xi_{-n_0-1}})^{-1}(q_1^1) - (f_{\xi_{-2n_0}\ldots\xi_{-n_0-1}})^{-1}(q) \fm  > \varepsilon.
\]
Thus,
$(f_{\xi_{-2n_0}\ldots\xi_{-n_0-1}})^{-1}(J_{i_0})$ is an interval of length greater or equal to $\varepsilon$ not intersecting $H_{\sigma^{-2n_0}(\xi)}$.

If we were in Case 1, we now repeat the same process for the points $g^{-1}(q_i)$ instead of $q_i$, $i =1,\ldots, M_0$, and if  we were in Case 2, then we do so for 
$q_i^1$, $i =1,\ldots, M_1$. 

This way, repeat this process, we will be in Case 1 at most $m_1$ times for some $m_1<n_1$. Indeed, otherwise if $n_1$ consecutive times occurs Case 1, together with  \eqref{contradicao1} we  would obtain
\begin{equation}\label{eq:expansao2}
    \am (f_{\xi_{-n_1n_0}\ldots\xi_{-1}})^{-1}(q_1) - (f_{\xi_{-n_1n_0}\ldots\xi_{-1}})^{-1}(q_{M_0}) \fm  \geq \Big(\frac{1}{\lambda_2}\Big)^{n_1}\am q_1 - q_{M_0} \fm
    >1.
\end{equation}
Hence, there is $m_1<n_1$ such that at the $m_1$th repetition of the process we are the first time in Case 2
and we let $q_1^{m_1} <\ldots < q_{M_1}^{m_1}$, with $M_1 = \left\lceil M_0/N_1 \right\rceil$, be an enumeration of the points in $H_{\sigma^{-m_1n_0}(\xi)}$ as above (note again that $M_1 \geq 2$ by the choice of $M_0$ in \eqref{eq:escdeM_0}).

Repeating the same arguments, we can recursively define positive integers $m_j$ and $n_j$, $j\ge1$, satisfying
\begin{equation}\label{contradicao2}
    \Big(\frac{1}{\lambda_2}\Big)^{n_j} \am q_{M_1}^{m_{j-1}} - q_1^{m_{j-1}} \fm > 1.
\end{equation}
 such that in this process $m_j$ consecutive times there will occur Case 1 followed by Case 2, each time defining points $q_1^{m_j} <\ldots < q_{M_j}^{m_j}$ in $H_{\sigma^{-m_jn_0}(\xi)}$,  $M_j = \lceil M_0/N_1^j \rceil$.

 Because, each time we are in Case 2, those intervals of length at least $\varepsilon$ are maintained and since the sum of their lengths must be smaller than $1$, by \eqref{eq:escdeM_0} this case occurs at most $N_1$ times and hence this process  must remain in Case 1 only. However,  by the above, this is also impossible.  So, we arrive at a contradiction.
This concludes the proof of the theorem.
\hfill\qed


\subsection{Proof of Theorem \ref{prop:medhip}}

To start, we recall the following classical concept and result (see \cite{A,ABV}, for example).

\begin{definition}
Given a number $\varrho > 0$ and a point $(\xi,p) \in \Sigma_N \times I$,  a natural number $n$ is a \emph{hyperbolic time for $(\xi,p)$ (with respect to $F$) with exponent $\varrho$} if
\[
    \am (f_{\xi_0\ldots\xi_n})'(p) \fm \geq e^{(n+1)\varrho} \quad \text{and}
\]
\[
    \am (f_{\xi_m\ldots\xi_n})'(f_{\xi_0\ldots\xi_{m-1}}(p)) \fm \geq e^{(n-m+1)\varrho} \quad\text{for all} \,\,m=1,\ldots,n.
\]
\end{definition}

We define the \emph{forward Lyapunov exponent of $(\xi,p)$ (with respect to $F$)} by
\[
    \chi_+(\xi,p):= \lim_{n\to +\infty}\frac{1}{n} \log \am (f_{\xi_0  \ldots  \xi_{n-1}})'(p) \fm.
\]
whenever the limit exists.
The following is a consequence of the Lemma of Pliss (see \cite{M}, for example).

\begin{lemma}\label{densposthipb}
If $\chi_+(\xi,p) > \chi - \varepsilon > 0$ then $(\xi,p)$ has infinitely many hyperbolic times with exponent $\chi - \varepsilon$. Moreover, the density of hyperbolic times is positive, that is, there exists $d_0 > 0$ such that
\[
    \liminf_{n \to +\infty} \frac{1}{n} \# \ak 1 \leq k \leq n \colon k \,\,\text{is a hyperbolic time for}\,\,(\xi,p) \fk \geq d_0.
\]
\end{lemma}

We also recall the concept of distortion.

\begin{definition}
Let $g\colon J \to I$, $J\subset I$ a compact interval, be differentiable  with $g'(x) \neq 0$ for all $x \in J$. For  $Z \subset J$, the \emph{maximal distortion of $g$ in $Z$} is defined by
\[
    \Dist g|_Z := \displaystyle\sup_{x,y \in Z} \frac{\am g'(x) \fm}{\am g'(y) \fm}.
\]
Given $\vartheta > 0$, let
\[
    D(\vartheta) := \displaystyle\max_{x \in I} \displaystyle\max_{i=1,\ldots,N} \ak \Dist f_i |_{\ac x - {\vartheta}/{2},x + {\vartheta}/{2}\fc \cap I}, \Dist f_i^{-1} |_{\ac x - {\vartheta}/{2},x + {\vartheta}/{2}\fc \cap I} \fk.
\]
\end{definition}
Note that $D(\vartheta) \rightarrow 1$ when $\vartheta \rightarrow 0$.

We  denote by $\am J \fm$ the length of an interval $J \subset I$. We borrow the following result on distortion control.

\begin{lemma}[{\cite[Lemma 4.1]{DGR2}}]\label{lem:contraminv}
Given $(\xi,p) \in \Sigma_N \times I$ such that $\chi=\chi_+(\xi,p) > 0$, let $\varepsilon > 0$ be such that $\chi - 2 \varepsilon > 0$. Let $n \geq 1$ be a hyperbolic time for $(\xi,p)$ with exponent $\chi - \varepsilon$, $\vartheta > 0$ small enough such that $D(\vartheta) < e^{\varepsilon}$, and $J_{n+1} \subset I$ an interval containing $f_{\xi_0\ldots\xi_n}(p)$ such that $\am J_{n+1} \fm \leq \vartheta$ and such that $(f_{\xi_k\ldots\xi_n})^{-1}|_{J_{n+1}}$ is well defined for all $k=0,\ldots,n$. Then we have
\[
    \am J_k \fm \leq \vartheta e^{-(n+1-k)(\chi - 2 \varepsilon)} \quad \text{for all}\,\, k=0,\ldots,n+1,\quad \text{where}\,\, J_k := (f_{\xi_k\ldots\xi_n})^{-1}(J_{n+1}).
\]
\end{lemma}

We are now ready to start the proof of Theorem \ref{prop:medhip}. We will prove the theorem assuming that $\chi(\mu) > 0$, the other case is analogous and will be omitted.

To sketch the proof, we will show  that $F^{-1}$ is contractive in a neighborhood in the fiber of points from $H$ and then $F$ expands in some neighborhood (in the fiber) of these points. The compactness of $I$ does not allow the existence of many points in the intersection of $H$ with each fiber. The others statements of the theorem follow from the monotonicity of the fiber maps.

Let $H$ be a $F$-invariant set such that $\mu(H) = 1$.  Without loss of generality, we can assume that every $(\xi,p)\in H$ satisfies $\chi_+(\xi,p) =\chi:= \chi(\mu) > 0$.
Given $(\xi,p) \in H$, define
\[
    A_{n,(\xi,p)}:=\ak 1 \leq k \leq n \colon k \,\,\text{is hyperbolic time for}\,\,(\xi,p)\,\,\text{with exponent}\,\,\chi(\mu) - \varepsilon \fk.
\]
Choosing $\varepsilon \in(0, \chi(\mu)/2)$, by Lemma \ref{densposthipb} there exists $d_0 > 0$ such that for all $(\xi,p) \in H$ we have
\[
    \liminf_{n \to +\infty} \frac{1}{n} \# A_{n,(\xi,p)} \geq d_0.
\]

Fix $\vartheta > 0$ such that $D(\vartheta) < e^\varepsilon$ and $N_1 \in \natur$ such that
\begin{equation}\label{eq:escdeN1}
    (N_1-1) \vartheta > 1.
\end{equation}
Fix also $0 < d_1 < d_0$ and let $j$ be a natural number such that
\begin{equation}\label{eq:escdel}
    j d_1 > N_1 -1 + d_1.
\end{equation}

By contradiction, let us suppose that for each $M \geq 1$ there exists $\xi \in \Sigma_N$ such that $\# H_\xi > M$.
Hence, for $j$ chosen as above, there exists $\xi \in \Sigma_N$ such that $\# H_\xi > j$. By $F$-invariance of $H$ we hence have $\# H_{\sigma^k(\xi)} > j$ for all $k \in \mathbb{Z}$.

\begin{claim}\label{claim:hyptim}
For every choice of points $x_1,\ldots,x_j$ in $H_\xi$ there is a subset $\ak x_{i_1},\ldots,x_{i_{N_1}} \fk \subset \ak x_1,\ldots,x_j \fk$ of points  which have infinitely many  common  hyperbolic times.
\end{claim}

\begin{proof}
Given $\ak x_1,\ldots,x_j \fk \subset H_\xi$, for every $i=1,\ldots,j$ there is $n_0$ such that for every $n\ge n_0^i$ we have
\[
	\frac1n\#A_{n,(\xi,x_i)}>d_1.
\]
Hence, every $x_i$ has at least $\left\lceil d_1n \right\rceil$ hyperbolic times between $1$ and $n$.

 As $jd_1 > N_1-1$, by \eqref{eq:escdel} we have $jd_1n > (N_1-1)n$ and, by the Pigeonhole Principle,  there exists $1 \leq t(n) \leq n$ such that at least $N_1$ of those points, say $x_{i_1}(n),\ldots,x_{i_{N_1}} (n)$, which have $t(n)$ as common hyperbolic time.

It remains to show that $t(n)$ can be arbitrarily big as $n\to\infty$.
By contradiction, suppose that  there exists $c \geq 1$ such that for all $n \geq n_0$ and for any choice of $x_{i_1}(n),\ldots,x_{i_{N_1}} (n)$ and of $t(n)$, we have $t(n) \leq c$. As
\[
    \lim_{n \to \infty} \frac{(N_1-1)(n-c)}{d_1n-c} = \frac{N_1-1}{d_1},
\]
there is $n' \geq n_0$ such that $d_1n'-c>0$ and
\begin{equation}\label{eq:3}
    \frac{(N_1-1)(n'-c)}{d_1n'-c} < \frac{N_1-1}{d_1} + 1 < j,
\end{equation}
where the last inequality follows from the choice of $j$ in \eqref{eq:escdel}. However, for each $i=1,\ldots,j$,  $x_i$ has at least $d_1n'-c$  hyperbolic times between $c+1$ and $n'$. Then it follows from \eqref{eq:3} and from Pigeonhole Priciple that there exists $c+1 \leq t^{*}(n') \leq n'$ such that at least $N_1$ points in $\ak x_1(n'),\ldots,x_j(n') \fk$ has $t^{*}(n')$ as a common hyperbolic time. As $t^{*}(n') > c$, we have a contradiction.

Finally, as there are only a finite number of choices of  $N_1$ points in $\ak x_1,\ldots,x_j\fk$, there exists at least one choice, we say $x_{i_1},\ldots,x_{i_{N_1}}$ which is repetitively chosen for infinitely many $n \geq n_0$ and such that $\lim_{n \to +\infty} t(n) = + \infty$.
\end{proof}

Let  $x_1,\ldots,x_{N_1} \in H_\xi$ be points with a arbitrarily large common hyperbolic times as in Claim~\ref{claim:hyptim}.
By choice of $N_1$ in \eqref{eq:escdeN1}, for each such a time $n$ there exist different points $y,z\in\ak x_{i_1}, \ldots, x_{i_{N+1}} \fk$ such that $\am f_{\xi_0\ldots\xi_{n}}(y) - f_{\xi_0\ldots\xi_{n}}(z) \fm < \vartheta$ for which, by Lemma \ref{lem:contraminv}, we obtain that
\[
    \am y-z \fm < \vartheta e^{-(n+1)(\chi(\mu) - 2 \varepsilon)}.
\]
As there is only a finite number of choices for $y$ and $z$, there exists  one which repeats infinitely often as $n\to\infty$. However, the latter inequality implies that $y=z$ leading to a contradiction.

This proves that there exists $M \geq 1$ such that $\# H_\xi \leq M$ for all $\xi \in \Sigma_N$.

Finally, for each $\xi \in \Pi_1(H)$ define  $a_\xi := \min H_\xi$ and $b_\xi := \max H_\xi$ and  consider
\[
    A := \ak (\xi,a_\xi) \colon \xi \in \Pi_1(H) \fk \quad \text{and} \quad B := \ak (\xi,b_\xi) \colon \xi \in \Pi_1(H) \fk,
\]
which are measurable. Moreover, by monotonicity of each fiber map and by the $F$-invariance of $H$  we have that $A \cup B$ is a $F$-invariant set. As $\mu$ is ergodic,  either $\mu(A \cup B) = 0$ or $\mu(A \cup B) =1$. If $\mu(A \cup B) =1$, it suffices to define
\[
    \varphi^- \colon\Pi_1(H) \to I,\quad \xi \mapsto a_\xi
    \quad\text{ and }\quad
        \varphi^+ \colon \Pi_1(H) \to I, \quad \xi \mapsto b_\xi.
\]
Otherwise, if $\mu(A \cup B) = 0$, we consider the set $H \setminus (A \cup B)$, which has full measure $\mu$ and is also invariant
repeat the previous analysis. As $\# H_\xi \leq M$ for all $\xi \in \Pi_1(H)$, at some point this process will finish. This concludes the proof of the theorem.
\qed

\begin{remark}{\rm
Note that it can occur that $(\Pi_1)_\ast\mu$-almost everywhere $\varphi^- = \varphi^+$.
On the other hand, by ergodicity of $\mu$, $\# H_\xi$ is constant $(\Pi_1)_*\mu$-almost everywhere and hence, by injectivity of the fiber maps, we either have $\varphi^- = \varphi^+$   or $\varphi^- < \varphi^+$  almost everywhere.
}\end{remark}

\begin{definition}
We say that $\mu$ has a \emph{simple graph} if $\varphi^- = \varphi^+$ in almost every point. In the other case, we say that $\mu$ has a \emph{non-simple graph}.
\end{definition}

\begin{corollary}
If $\mu$ has a non-simple graph then $\Gamma^-$ and $\Gamma^+$ both have positive measure and are both non-invariant.
\end{corollary}

Ending this section let us see some examples. 

\begin{example*}
Suppose that all fiber maps $f_i$ preserve orientation.
Let $\mu$ be a hyperbolic ergodic measure. Then, by Theorem \ref{prop:medhip}, $\mu$ has a simple graph. In fact, in this case the sets $\Gamma^{-}$ and $\Gamma^{+}$ from the statement of the theorem are $F$-invariant and one of them has full measure.

Consider  a hyperbolic periodic orbit $\mathcal{O}(\xi,p)$ of a periodic point with period $n$ such that $\am (f_{\xi_0\ldots\xi_{n-1}})'(p) \fm \neq 1$ which is a hyperbolic set satisfying $ \# H_\omega = 1$ for all $\omega \in \Pi_1(\mathcal{O}(\xi,p))$.
\end{example*}

\begin{example*}[Bi-graph]
The simplest example of a bi-graph occurs if the step skew-product has two fiber maps $f_1$ and $f_2$ such that $f_1$ preserves orientation and $f_2$ reverses and such that there exist points $0 < p_1 < p_2 < 1$ satisfying
\[
        f_1(p_1) = p_1,\, f_1(p_2) = p_2,\, f_2(p_1) = p_2 \,\,\text{and}\,\, f_2(p_2) = p_1.
\]
Then $((1212)^\mathbb{Z},p_1)$ is periodic with period $4$. If it is hyperbolic then the orbit $\mathcal{O}((1212)^\mathbb{Z},p_1)$ is an example of hyperbolic set with two points in each fiber it intersects. Moreover, if we consider $\mu$ the average of the Dirac measures supported in the points from $\mathcal{O}((1212)^\mathbb{Z},p_1)$
then we have that the set $H$ from Theorem \ref{prop:medhip} (chosen with respect to $\mu$) must be $\mathcal{O}((1212)^\mathbb{Z},p_1)$, that is, we have an example which the set $H$ is a bi-graph and the measure has a non-simple graph.
\end{example*}


\section{Topological attractors in a generic step skew-product}\label{sec:dat}

In this section we prove Theorem \ref{teo:TP}.
The general strategy (already indicated in \cite{KV}) is to consider an appropriate extension of the step skew-product $F$ introduced in Definition \ref{defdeF}, obtaining a new skew-product for which all fiber maps preserve orientation and then to apply \cite{KV}.
The results for the extended step skew-product will be carried over to $F$ by a semi-conjugation.

The topology considered in the space of step skew-products as well as the open and dense subset stated in Theorem \ref{teo:TP} will be described in Section ~\ref{subsec:cg}.
The definition of bi-strips is given in Section \ref{subsec:far} and the proof of Theorem \ref{teo:TP} will be completed in Section \ref{subsec:demTP}.

In this section we consider  a step skew-product $F$ as in Definition \ref{defdeF} and also  suppose that $f_i(I) \subset \ine(I)$ for all $i=1,\ldots,N$. To exclude the case already covered in~\cite{KV}, we will suppose $F$ has at least one fiber map reversing orientation.

\begin{remark}{\rm
	Theorem~\ref{teo:TP} can be extended to the more general case in which $\lambda_0$ is a quasi-Bernoulli measure%
\footnote{A measure $\lambda_0$ is said \emph{quasi-Bernoulli} if there exists $C \geq 1$ such that for all $l, m, n \in \mathbb{Z}$ with $l \leq m \leq n$ is valid that $C^{-1} \lambda_0([l;\omega_l\ldots\omega_{n}]) \leq \lambda_0([l;\omega_l\ldots\omega_m])\lambda_0([m;\omega_{m}\ldots\omega_{n}]) \leq C \lambda_0([l;\omega_l\ldots\omega_{n}])$ (see \eqref{eq:cilindro} for the precise definition of $[l;\omega_l\ldots\omega_n]$). It is not difficult to see that any Bernoulli measure, any Markov measure and any Gibbs measure are quasi-Bernoulli measures.}.
	In fact, as the proof of Theorem \ref{teo:TP} is based on the results from~\cite{KV}, it is sufficient to verify that those results are also valid for quasi-Bernoulli measures. For such, we observe that the Markov property is only used in the proof in~\cite[page 20, 6.13 Proposition]{KV}, which is also valid in the quasi-Bernoulli case.
	However, to simplify the exposition, we will restrict to the Markov case.
}\end{remark}

\subsection{Extended step skew-product}\label{subsec:ESSP}
We  will construct a new step skew-product which  is semi-conjugated to the first one and which has the property that all fiber maps preserve orientation.
Denote by
\[
    \mathcal{I}_P := \ak i \colon f_i \,\,\text{preserves orientation} \fk \quad \text{and} \quad  \mathcal{I}_R := \ak i \colon f_i \,\,\text{reverses orientation} \fk
\]
and consider the associated transition matrix $A = \big(a_{ij}\big)_{i,j = 1}^{2N}$ given by
\[
    a_{ij} := \left\{
                \begin{array}{ll}
                  1, \, & \hbox{if}\,\, i \in \mathcal{I}_P\,\,\text{and}\,\, j \in \ak 1,\ldots,N \fk, \,\mbox{or}\\
                   & \hbox{if}\,\, i \in \mathcal{I}_R \,\,\text{and}\,\, j \in \ak N+1,\ldots,2N \fk, \,\mbox{or}\\
                   & \hbox{if}\,\, i-N \in \mathcal{I}_P\,\,\text{and}\,\, j \in \ak N+1,\ldots,2N \fk, \,\mbox{or}\\
                   & \hbox{if}\,\, i-N \in \mathcal{I}_R \,\,\text{and}\,\, j \in \ak 1,\ldots,N \fk; \\
                  0, & \hbox{otherwise}.
                \end{array}
              \right.
\]
Denote by $\Sigma_A \subset \Sigma_{2N}$ the set of $A$-admissible sequences and by $\sigma_A\colon \Sigma_A \to \Sigma_A$ the shift map. Note that $\sigma_A$ is transitive.

To relate the shift maps $\sigma_A\colon\Sigma_A\to\Sigma_A$ and $\sigma\colon\Sigma_N \to \Sigma_N$, for each $i = 1,\ldots,2N$, write $\overline{i} := i \mod N$ and define the projection map
\[
    \pi \colon \Sigma_A \to \Sigma_N \quad \colon \quad \omega \mapsto \overline{\omega}, \quad \text{where}\,\, \,\overline{\omega}_n := \overline{\omega_n}\, \,\, \text{for all}\,\, n \in \mathbb{Z}.
\]

\begin{lemma}\label{lem:conj}
The projection $\pi$ is a semi-conjugation two-to-one between $\sigma_A$ and $\sigma$, that is, $\pi$ is continuous, surjective,
$\pi \circ \sigma_A = \sigma \circ \pi$ and each point in $\Sigma_N$ has exactly two pre-images by $\pi$ in $\Sigma_A$.
\end{lemma}

\begin{proof}
It is not difficult to see that $\pi$ is continuous and $\pi \circ \sigma_A = \sigma \circ \pi$.
To show that $\pi$ is surjective and two-to-one, given $\xi \in \Sigma_N$, define the sequence $\omega$ by
\[\begin{split}
	\omega_0 :=\xi_0,\quad
    \omega_n &:=  \Big\{
                   \begin{array}{ll}
                     \xi_n, \, & \hbox{if}\,\, a_{\omega_{n-1}\xi_n} = 1 \\
                     \xi_n + N, \,& \hbox{if}\,\, a_{\omega_{n-1}(\xi_n+N)} = 1
                   \end{array}
     \quad\text{ for every }n\ge1          ,  \\
    \omega_n &:=  \Big\{
                   \begin{array}{ll}
                     \xi_n, \, & \hbox{if}\,\, a_{\xi_n\omega_{n+1}} = 1 \\
                     \xi_n + N, \,& \hbox{if}\,\, a_{(\xi_n+N)\omega_{n+1}} = 1
                   \end{array}
                 \quad\text{ for every }n\le1.
\end{split}\]
Note that $\omega\in\Sigma_A$ and $\pi(\omega) = \xi$, proving that $\pi$ is surjective.
Note that the sequence $\eta$ such that $\eta_0=\xi_0+N$ and which is defined elsewhere as above with $\eta_n$ in the place of $\omega_n$ for all $n\ne0$ also satisfies $\eta\in\Sigma_A$ and $\pi(\eta)=\xi$.
Finally, observe that by definition of $A$, those sequences above are the only ones having these properties, proving that $\pi$ is two-to-one.
\end{proof}

\begin{remark}{\rm
Note that $\pi$ provides only a semi-conjugation. Indeed, denoting
\begin{equation}\label{eq:defdeC}
    C := \ak \omega \in \Sigma_A \colon \omega_0 \in \{ 1,\ldots,N \} \fk,
\end{equation}
the proof of Lemma \ref{lem:conj} implies that $\pi|_C$ e $\pi|_{\Sigma_A \setminus C}$ are both bijections onto $\Sigma_N$, but it is not true that $\sigma_A(C) \subset C$ or $\sigma_A(\Sigma_A \setminus C) \subset \Sigma_A \setminus C$.
}\end{remark}

\begin{definition}\label{defdeG}
Consider the map 
\[
    R\colon I \to I \quad \colon \quad x \mapsto 1-x.
\]
Let $F$ be a step skew-product as introduced in Definition \ref{defdeF}. We define the \emph{extended step skew-product (with respect to F)} as being the map
\[
    G\colon \Sigma_A \times I \to \Sigma_A \times I \quad \colon \quad (\omega,x) \mapsto (\sigma_A(\omega), g_{\omega_0}(x)),
\]
where
\begin{equation}\label{eq:deffibermaps}\begin{split}
    g_i = f_i \quad \text{and} \quad g_{i+N} = R\circ f_i \circ R \quad
    &\text{for each}\,\, i \in \mathcal{I}_P,
\\
    g_i = R \circ f_i \quad \text{and} \quad g_{i+N} = f_i \circ R \quad
    &\text{for each}\,\, i \in \mathcal{I}_R.
\end{split}\end{equation}
\end{definition}

In the remainder of this section (except for a part of Section \ref{subsec:cg}), the step skew-product $F$ will be fixed and we always will consider $G$ being the extended step skew-product (with respect to $F$).

Consider the projection
\[
    \Pi \colon \Sigma_A \times I \to \Sigma_N \times I \quad \colon \quad (\omega,x) \mapsto \left\{
                                                                                        \begin{array}{ll}
                                                                                          (\overline{\omega},x), \, & \hbox{if}\,\,\omega \in C \\
                                                                                          (\overline{\omega},R(x)),\,& \hbox{if}\,\,\omega \in \Sigma_A \setminus C,
                                                                                        \end{array}
                                                                                       \right.
\]
where $C$ is as in \eqref{eq:defdeC}. Note that $\Pi|_{C \times I}$ and $\Pi|_{(\Sigma_A \setminus C) \times I}$ are bijections onto $\Sigma_N \times I$. Denote by $\Pi_1\colon\Sigma\times I\to\Sigma$ and $\Pi_2\colon\Sigma\times I\to I$ the canonical projections onto the first and second coordinate, respectively. Analogously, $\Pi_1^A$ and $\Pi_2^A$.
Compare the following diagram.
\[
\xymatrix@1{
& \Sigma_A\times I\ar@{->}[rr]^-{\quad\quad\quad G}\ar@{->}'[d][dd]
    && \Sigma_A\times I \ar@{->}[dd]^-{\Pi}
\\
\Sigma_A \ar@{<-}[ur]^-{\Pi_1^A}\ar@{->}[rr]^-{\quad\quad \sigma_A}\ar@{->}[dd] _-{\pi}
    &&\Sigma_A \ar@{<-}[ur]^-{\Pi_1^A}\ar@{->}[dd]
\\
& \Sigma_N\times I\ar@{->}'[r][rr]^-{F}
    && \Sigma_N\times I
\\
\Sigma_N\ar@{->}[rr]^-{\sigma_N}\ar@{<-}[ur]^-{\Pi_1}
    &&\Sigma_N \ar@{<-}[ur]^-{\Pi_1}
}
\]

\vspace{0.5cm}

\begin{lemma}\label{lem:conj2}
The projection $\Pi$ is a semi-conjugation two-to-one between $G$ and $F$.
\end{lemma}

\begin{proof}
Lemma \ref{lem:conj} implies that $\Pi$ is continuous, surjective and two-to-one.
It remains only to show that $\Pi\circ G = F \circ \Pi$.
Fix $(\omega,x) \in \Sigma_A \times I$.
We will only study the case $\omega \in C$ and $\omega \in \mathcal{I}_P$, the other cases are analogous and  omitted.
In this case $g_{\omega_0} = f_{\omega_0}$ and then
\[
    \Pi \circ G((\omega,x)) = \Pi (\sigma_A(\omega), f_{\omega_0}(x)).
\]
Furthermore,  $\sigma_A(\omega) \in C$ and by the above equality  and the definition of $\Pi$ we have
\[
    \Pi \circ G((\omega,x)) = (\overline{\sigma_A(\omega)}, f_{\omega_0}(x)).
\]
Using again Lemma \ref{lem:conj} and that $\omega \in C$, the last equality implies that
\[
    \Pi \circ G((\omega,x)) = (\sigma(\overline{\omega}), f_{\overline{\omega_0}}(x)) = F((\overline{\omega},x)) = F \circ \Pi((\omega,x))
\]
proving the lemma.
\end{proof}

\begin{remark}{\rm
It is an immediate consequence of Lemma \ref{lem:conj2} and \cite[Chapter IX, Theorem 1.8]{R} that the topological entropy of $F$ and $G$ coincide.
}\end{remark}

Recall the definition of $\Lambda$ being the maximal invariant set of $F$  in~\eqref{defdeLambda} and consider the analogously defined \emph{maximal invariant set} of $G$
\[
    \Gamma := \bigcap_{n \geq 0} G^n(\Sigma_A \times I).
\]
Note that we can write
\[
    \Gamma = \bigcup_{\omega \in \Sigma_A} \ak \omega \fk \times I_\omega,
    \quad\text{ where }\quad
    I_\omega := \bigcap_{n \geq 1} g_{\omega_{-n}\ldots \omega_{-1}} (I).
\]
There is an analogous  decomposition also for $\Lambda$.

\begin{lemma}\label{lem:simatrator}
We have   $I_\omega = I_{\overline{\omega}}$ for each $\omega \in C$ and $I_\omega = R(I_{\overline{\omega}})$ for each $\omega \in \Sigma_A \setminus C$.
\end{lemma}

\begin{proof}
Fix $\omega \in C$. It suffices to show that $g_{\omega_{-1}}\circ \ldots \circ g_{\omega_{-n}}(I) = f_{\overline{\omega_{-1}}}\circ \ldots \circ f_{\overline{\omega_{-n}}}(I)$ for all $n \geq 1$. Note first that from $\omega\in C$  together with Lemma \ref{lem:conj2} for any $n\ge1$ we obtain \[\begin{split}
    g_{\omega_{-1}}\circ \ldots \circ g_{\omega_{-n}}(I)
    &= \Pi_2^A (G^n(\ak \sigma^{-n}(\omega) \fk \times I))\\
    &= \Pi_2^A ((\Pi|_{C \times I})^{-1} (F^n(\Pi(\ak \sigma^{-n}(\omega) \fk \times I)))).
\end{split}\]
Using that $R(I) = I$, the definitions of $\Pi$, $F$,  $(\Pi|_{C \times I})^{-1}$, and $\Pi_2^A$ then imply
\[
\begin{split}
    g_{\omega_{-1}}\circ \ldots \circ g_{\omega_{-n}}(I)
    &= \Pi_2^A ((\Pi|_{C \times I})^{-1} (F^n (\ak \overline{\sigma^{-n}({\omega})}\fk \times I)))\\
    &= \Pi_2^A ((\Pi|_{C \times I})^{-1} (\ak \overline{\omega}\fk \times (f_{\overline{\omega_{-1}}}\circ \ldots \circ f_{\overline{\omega_{-n}}}(I))))\\
    &= \Pi_2^A (\ak \omega \fk \times (f_{\overline{\omega_{-1}}}\circ \ldots \circ f_{\overline{\omega_{-n}}}(I)))\\
    &= f_{\overline{\omega_{-1}}}\circ \ldots \circ f_{\overline{\omega_{-n}}}(I)
\end{split}
\]
proving the claimed property.
The case $\omega \in \Sigma_A \setminus C$ is analogous.
\end{proof}


\subsection{Strips, attractors and repellers}\label{subsec:far}

In this section we will define the concepts of attracting (repelling) (bi-)strips in a step skew-product. We will also show that attracting (repelling) strips associated to the extended step skew-product $G$ correspond to attracting (repelling) bi-strips associated to $F$.

From now on, we will write $\Sigma$ when referring to any of the spaces $\Sigma_A$ or $\Sigma_N$.

\begin{definition}\label{def:duplafaixa}
Given two functions $\varphi, \psi\colon \Sigma \to I$, we say that \emph{$\varphi$ is smaller than $\psi$} and   write $\varphi < \psi$ whenever
\[
    \varphi(\omega) < \psi(\omega) \quad \text{for all}\,\, \omega \in \Sigma.
\]
If $\varphi < \psi$, we define \emph{the strip between the graphs of $\varphi$ and $\psi$} by
\[
    S_{\varphi,\psi} := \ak (\omega, x)  \in \Sigma \times I \colon \varphi(\omega) \leq x \leq \psi(\omega) \fk.
\]
A strip $S_{\varphi,\psi}$ is \emph{attracting} with respect to a step skew-product $H \colon \Sigma \times I \to \Sigma \times I$ if
\[
    H(S_{\varphi,\psi}) \subset \ine(S_{\varphi,\psi}),
\]
and \emph{repelling} with respect to $H$  (provided the inverse  of $H$ is defined) if
\[
    H^{-1}(S_{\varphi,\psi}) \subset \ine(S_{\varphi,\psi}).
\]
A subset $S \subset \Sigma \times I$ is said a \emph{bi-strip} if $S = S_{\varphi_1,\psi_1} \cup S_{\varphi_2,\psi_2}$, where $S_{\varphi_j,\psi_j}$ are strips between the graphs of $\varphi_j$ and $\psi_j$ for $j=1,2$. Analogously, we say that a bi-strip is \emph{attracting} (\emph{repelling}) with respect to $H$ if it satisfies the above properties.
\end{definition}

Note that in the definition of bi-strip we do not require the condition $\psi_1 < \varphi_2$ or $\psi_2 < \varphi_1$. Although, the systems we study will have one of these properties.

There is the following relation between strips of  $G$ and   bi-strips of $F$.

\begin{lemma}\label{projdasfaixas}
If $S$ is an attracting strip with respect to  $G$ then $\Pi(S)$ is an attracting bi-strip with respect to $F$. Analogously, if $S$ is a repelling bi-strip with respect to $G$ then $F^{-1}$ is well defined in $\Pi(S)$ and $\Pi(S)$ is a repelling bi-strip with respect to $F$.
\end{lemma}

\begin{proof}

Let $S$ be a strip in $\Sigma_A \times I$. Using that
\[
     \Pi(B_\varepsilon(\omega) \times B_\varepsilon(x)) = \left\{
                                                     \begin{array}{ll}
                                                      B_\varepsilon(\overline{\omega}) \times B_\varepsilon(x), \, & \hbox{if}\,\, (\omega,x) \in C \times I \\
                                                      B_\varepsilon(\overline{\omega}) \times R(B_\varepsilon(x)), \, & \hbox{if}\,\, (\omega,x) \in (\Sigma_A \setminus C) \times I
                                                      \end{array}
                                                     \right.
\]
for each $0 < \varepsilon < 1$, one can show that
\begin{equation}\label{eq:antigaaf2}
    \Pi(\ine(S)) \subset \ine (\Pi(S)).
\end{equation}
Suppose $S$ is attracting with respect to $G$. Then, by Lemma \ref{lem:conj2} and \eqref{eq:antigaaf2},
\[
    F(\Pi(S)) = \Pi(G(S)) \subset \Pi(\ine(S)) \subset \ine(\Pi(S)).
\]
Suppose now $S$ is repelling with respect to $G$. Fix $(\xi,z) \in \Pi(S)$ and $(\omega,x) \in S$ such that $(\xi,z) = \Pi((\omega,x))$.
As $S$ is repelling, we have $(\omega,x) \in G(\Sigma_A \times I)$ and
\[
    (\xi,z) = \Pi((\omega,x)) \in \Pi(G(\Sigma_A \times I)) = F(\Pi(\Sigma_A \times I)) = F(\Sigma_N \times I).
\]
This shows that $F^{-1}$ is well defined in $\Pi(S)$. The proof that $\Pi(S)$ is repelling is analogous to the previous case.
\end{proof}

Let $S \subset \Sigma \times I$ be a (bi-)strip. If $S$ is attracting with respect to a step skew-product $H$, we define the \emph{maximal attractor of $S$} as the set
\[
    A_{\max}(S) := \displaystyle\bigcap_{n\geq 0} H^n(\Sigma \times I).
\]
Analogously, if $S \subset \Sigma \times I$ is repelling with respect to $H$, we define the \emph{maximal repeller of $S$} as the maximal attractor of $S$ with respect to $H^{-1}$.

\begin{lemma}\label{lem:relentreatratores}
If $S$ is an attracting (repelling) strip with respect to $G$ and $B$ is its maximal attractor (repeller), then $\Pi(B)$ is the maximal attractor (repeller) of the bi-strip $\Pi(S)$ (attracting (repelling) with respect to F).
\end{lemma}

\begin{proof}
Suppose $S$ is attracting with respect to $G$. Using that $\{ G^n(S) \}_{n \geq 0}$ is a nested collection of closed strips, the compactness of $\Sigma_A \times I$ and the continuity of $\Pi$ imply
\[
    \bigcap_{n \geq 0} \Pi(G^n(S)) = \Pi\Big(\bigcap_{n \geq 0} (G^n(S))\Big).
\]
Combining Lemma \ref{lem:conj2} with the last equality, we obtain that
\[
    \bigcap_{n \geq 0} F^n(\Pi(S)) = \bigcap_{n \geq 0} \Pi(G^n(S)) = \Pi\Big(\bigcap_{n \geq 0} (G^n(S))\Big) = \Pi(B).
\]
This proves that $\Pi(B)$ is the maximal attractor of $\Pi(S)$. The proof in the case $S$ repelling is analogous.
\end{proof}


\subsection{Extended Markov measures}\label{subsec:me}

In this section we will compare measures in the spaces $\Sigma_A$ and $\Sigma_N$. We will also compare the existence of (bi-)bony graphs in the spaces $\Sigma_A \times I$ and $\Sigma_N \times I$. Finally, we will show that a physical hyperbolic measure with respect to the extended step skew-product $G$ projects to a physical hyperbolic measure with respect to $F$.

For each $n \geq m \in \mathbb{Z}$ and $\omega_m\ldots\omega_n$ finite sequence, we will denote the \emph{cylinder associated to $\omega_m\ldots\omega_n$} by
\begin{equation}\label{eq:cilindro}
    [m;\omega_m\ldots\omega_n] := \ak \eta \in \Sigma \colon \eta_i = \omega_i \,\,\text{for all} \,\,i = m,\ldots,n \fk.
\end{equation}
Given a Markov measure in $\lambda$ in $\Sigma_A$ (see, for example, \cite[Chapter 1]{W}), we will denote by $(p_1,p_2,\ldots,p_{2N})$ the probability vector  and by $(P_{ij})_{i,j=1}^{2N}$ the stochastic matrix associated to $\lambda$.

\begin{definition}\label{def:medsimetrica}
A Markov measure $\lambda$ in $\Sigma_A$ is \emph{symmetric} if $p_i = p_{i+N}$, $P_{ij} = P_{(i+N)(j+N)}$ and $P_{i(j+N)} = P_{(i+N)j}$
for all $i,j = 1,\ldots,N$.
\end{definition}

Note that by definition of $A$, if $\lambda$ is a Markov measure in $\Sigma_A$ then for each $i,j = 1,\ldots,N$, we have $P_{ij} = 0$ or $P_{i(j+N)} = 0$.

\begin{definition}\label{defdelambdabarra}
Given a measure $\lambda$  in $\Sigma_A$, we denote
\[
\overline{\lambda} := \pi_* \lambda
\]
and call $\overline\lambda$ the \emph{projection} of $\lambda$ and  $\lambda$ an \emph{extension} of $\overline\lambda$. Given a Markov measure $\lambda_0$ in $\Sigma_N$ we call an extension of it $\lambda$ a \emph{symmetric extension} if  $\lambda_0=\pi_\ast\lambda$  and $\lambda$ is symmetric.
\end{definition}

The following result is immediate by construction.

\begin{lemma}\label{lem:unisymext}
	Every Markov measure $\lambda_0$ in $\Sigma_N$ has a unique symmetric extension.
\end{lemma}

We have the following correspondence, which in particular applies to Markov measures in $\Sigma_N$ and symmetric Markov measures in $\Sigma_A$.

\begin{lemma}\label{lem:af3teo}
We have $\pi\circ\Pi_1^A=\Pi_1\circ\Pi$. Moreover, if  $\lambda_0$ is a Borel probability in $\Sigma_N$, $\lambda$ an extension of $\lambda_0$ to $\Sigma_A$, and $\mu$ an extension of $\lambda$ to $\Sigma_A\times I$ which is ergodic $G$-invariant, then its projection $\Pi_*\mu$ to $\Sigma_N\times I$ is ergodic $F$-invariant  and we have $\lambda_0 = ({\Pi_1})_*(\Pi_*\mu)$.
\end{lemma}

We  now  compare  continuous (bi-)bony graphs (recall Definition \ref{def:bigraficoossudo}) in the spaces $\Sigma_A \times I$ and $\Sigma_N \times I$. Note that in the definition of a (bi-)bony graph we do not require that its projection onto $\Sigma$ was this whole set. However, we will work with (bi-)bony graphs satisfying this property. As in the case of strips (Lemma~\ref{projdasfaixas}), 
(continuous) bony graphs in $\Sigma_A \times I$ correspond to (continuous) bi-bony graphs in $\Sigma_N \times I$.

\begin{lemma}\label{lem:duplobg}
Let $\lambda$ be a symmetric Markov measure in $\Sigma_A$. If $B$ is a bony graph with respect to $\lambda$ then $\Pi(B)$ is a bi-bony graph with respect to $\overline{\lambda}=\pi_*\lambda$.
Moreover, if $B$ is continuous then $\Pi(B)$ is, too.
\end{lemma}

\begin{proof}
Recalling the definition of $C$ in~\eqref{eq:defdeC}, consider the partition
\[
	\Pi(B) = \Pi|_{C \times I} (B) \cup \Pi|_{(\Sigma_A \setminus C) \times I}(B).
\]
To prove that $\Pi(B)$ is a bi-bony graph, it suffices to show that $\Pi|_{C \times I} (B)$ and $\Pi|_{(\Sigma_A \setminus C) \times I}(B)$ are two \emph{bony} graphs. By definition of $\Pi$, $(\Pi|_{C \times I} (B))_\xi$ is either a point or an interval for every sequence $\xi \in \Sigma_N$.
Note that
\[
    \pi^{-1}(\ak \xi \in \Sigma_N \colon (\Pi|_{C \times I} (B))_\xi \,\,\text{is an interval} \fk) = \ak \omega \in C \colon B_\omega \,\,\text{is an interval} \fk.
\]
Since, by hypothesis, $B$ is a bony graph, we have
\[
    \lambda (\ak \omega \in C \colon B_\omega \,\,\text{is an interval}) \fk = 0
    =   \overline{\lambda} (\ak \xi \in \Sigma_N \colon (\Pi|_{C \times I} (B))_\xi\text{ is an interval} \fk ),
\]
where the latter is a consequence of the definition of $\overline\lambda$.
This shows that $\Pi|_{C \times I} (B)$ is a bony graph. The proof that $\Pi|_{(\Sigma_A \setminus C) \times I} (B)$ is a bony graph is analogous.
The  statement about continuity  is immediate.
\end{proof}

Ending this section we will show that a physical hyperbolic measure with respect to the extended step skew-product $G$ projects to a physical hyperbolic measure with respect to $F$. Recall its definition.

\begin{definition}\label{def:medidafisica}
Let $H\colon \Sigma \times I \to \Sigma \times I$ be a step skew-product and $\lambda$ a Markov measure in $\Sigma$. An $H$-invariant measure $\mu$ is called a \emph{physical measure with respect to $H$ and $\lambda$} provided that $(\lambda \times \Leb)(V) > 0$, where $\Leb$ denotes the Lebesgue measure and where $V$ is the set of all points $(\xi, p) \in \Sigma \times I$ satisfying
\[
    \lim_{n \to \infty} \frac{1}{n} \displaystyle\sum_{i=0}^{n-1} \varphi(H^i(\xi,p)) = \int \varphi \,d\mu \quad\text{for all}\,\,\varphi \in C^0(\Sigma \times I).
\]
The set $V$ is called  \emph{basin of the measure $\mu$}.
\end{definition}

\begin{lemma}\label{lem:medfisica}
Let $\lambda$ be a symmetric Markov measure in $\Sigma_A$ and $\lambda_0=\pi_*\lambda$ its projection on $\Sigma_N$. If $\mu$ is a physical measure with respect to  $G$ and $\lambda$ then $\Pi_\ast\mu$ is a physical measure with respect to $F$ and $\lambda_0$. Moreover, if $V$ is the basin of $\mu$ then $\Pi(V)$ is contained in the basin of  $\Pi_\ast\mu$.
\end{lemma}

\begin{proof}
Fix $(\xi,p) \in \Pi(V)$, $\varphi \in C^0(\Sigma_N \times I)$ and $(\omega,x) \in V$ such that $\Pi(\omega,x) = (\xi,p)$. By semiconjugation (Lema \ref{lem:conj2}) and continuity of $\varphi\circ \Pi$, we have
\[
    \lim_{n \to \infty} \frac{1}{n} \displaystyle\sum_{i=0}^{n-1} \varphi (F^i (\xi,p)) =\lim_{n \to \infty} \frac{1}{n} \displaystyle\sum_{i=0}^{n-1} (\varphi \circ \Pi) (G^i(\omega,x)) =\int \varphi \circ \Pi \,d\mu =\int \varphi \,d(\Pi_\ast\mu).
\]
Hence $\Pi(V)$ is contained in the basin of $\Pi_\ast\mu$. Now $\Pi_\ast(\lambda \times \Leb) = \lambda_0 \times \Leb$ implies
\[
    (\lambda_0 \times \Leb)(\Pi(V)) 
    = (\lambda \times \Leb)(\Pi^{-1}(\Pi(V))) \geq (\lambda \times \Leb) (V) > 0.
\]
This concludes the proof.
\end{proof}

Let us compare fiberwise Lyapunov exponents of points in $\Sigma_A \times I$ and $\Sigma_N \times I$.

\begin{lemma}\label{lem:expigual}
For every $(\omega,x) \in \Sigma_A \times I$ for which $\chi_+(\omega,x)$ (with respect to $F$) is well defined, we have $\chi_+(\omega,x) = \chi_+^G(\Pi(\omega,x))$, where $\chi_+^G$ denotes the fiberwise Lyapunov exponent with respect to $G$.
\end{lemma}

\begin{proof}
If $\omega \in C$ (recall the definition of $C$ in \eqref{eq:defdeC}) then $\Pi(\omega,x) = (\overline{w},x)$ and as
$R^{-1} = R$, it is not difficult to see that
\[
     f_{\overline{\omega_0}\ldots\overline{\omega_{n-1}}} = g_{\omega_0\ldots\omega_{n-1}} \quad \text{or} \quad f_{\overline{\omega_0}\ldots\overline{\omega_{n-1}}} = R \circ g_{\omega_0\ldots\omega_{n-1}} \quad \text{for each}\,\, n \geq 1.
\]
As $\am R'(y) \fm = 1 $ for all $y \in I$, it follows that in both cases
\[
    \frac{1}{n} \log \am (f_{\overline{\omega_0}\ldots\overline{\omega_{n-1}}})'(x) \fm = \frac{1}{n} \log \am (g_{\omega_0\ldots\omega_{n-1}})'(x) \fm.
\]
Thus, $\chi_+(\omega,x) = \chi_+(\Pi(\omega,x))$ for $\omega \in C$. The case $\omega \notin C$ is analogous.
\end{proof}

\begin{corollary}\label{lem:medhiperbolica}
For every  $G$-invariant and ergodic measure $\mu$ we have $\chi(\mu) = \chi(\Pi_\ast\mu)$. In particular, if $\mu$ is hyperbolic (with respect to $G$) then $\Pi_\ast\mu$ is hyperbolic (with respect to $F$).
\end{corollary}


\subsection{Stationary measures}\label{subsec:sm}

We will now consider (ergodic) stationary measures with respect to the extended step skew-product $G$. This is of importance, because their support is intimately related with  the bi-strips stated in Theorem \ref{teo:TP}. As above, we will establish a number of symmetry relations.

In the remainder of this section, $\lambda$ is always a symmetric Markov measure in $\Sigma_A$ with respect to $G$  and $(P_{ji})_{j,i=1}^{2N}$ denotes the stochastic matrix associated to $\lambda$.

\begin{definition}\label{def:medidaestacionaria}
A Borel probability $\mu$ in $\ak 1,\ldots,2N \fk \times I$ is \emph{stationary} (relative to $G$ and $\lambda$) if
\[
    \mu_i(E) = (g_*\mu)_i(E) := \sum_{j=1}^{2N} P_{ji} \mu_j(g_i^{-1}(E)) \quad \text{for all}\,\,E \subset I \,\,\text{Borel set and}\,\,i=1,\ldots,2N,
\]
where $\mu_i$ denotes the restriction $\mu|_{\{i\}\times I}$ (the fiber $\ak i \fk \times I$ is here naturally identified with the interval $I$).
\end{definition}

Let $A \subset (\ak 1,\ldots,2N \fk \times I)$. For each $i = 1,\ldots,2N$, we will denote $A_i := P_2(A \cap (\ak i \fk \times I))$, where $P_2\colon \ak 1,\ldots,2N \fk \times I \to I$ is the canonical projection.

\begin{definition}
A stationary measure $\mu$ in $\ak 1,\ldots,2N \fk \times I$ is  \emph{ergodic} if $\mu_i(A_i) = 1$ for all $i \in \ak1,\ldots,2N\fk$ whenever $A \subset (\ak 1,\ldots,2N \fk \times I)$ satisfies $g_j^{-1}(A_i) = A_j$ for all $i,j$ such that $P_{ji} > 0$.
\end{definition}

\begin{definition}\label{eq:defdemu'}
Let $s\colon \ak 1,\ldots,2N \fk \to \ak 1,\ldots,2N \fk$ given by
\[
       s(i) :=  \left\{
                  \begin{array}{ll}
                    i+N \,& \hbox{if}\,\,i \in \ak 1,\ldots,N \fk \\
                    i-N \, & \hbox{if}\,\,i \in \ak N+1,\ldots,2N \fk.
                  \end{array}
                \right.
\]
Given a Borel probability $\mu$ in $\ak 1,\ldots,2N \fk \times I$, the \emph{mirrored measure of $\mu$} is the measure $\mu'$ in $\ak 1,\ldots,2N \fk \times I$ defined by
\begin{equation}\label{eq:defmirrored}
    \mu'(D \times E) := \mu(s(D) \times R(E))
\end{equation}
for every $D \subset \ak 1,\ldots,2N \fk$, $E \subset I$ Borel sets and extended to the Borel sets from ${\ak 1,\ldots,2N \fk \times I}$ (by the Carath\'eodory Extension Theorem). 
\end{definition}

\begin{lemma}\label{ergdemu'}
The mirrored measure of any stationary (ergodic) measure (relative to $G$ and $\lambda$) is  stationary (ergodic).
\end{lemma}

\begin{proof}
Consider a Borel set $E\subset I$ and $i \in \ak 1,\ldots,N \fk$. By definition of being mirrored~\eqref{eq:defmirrored} and by stationarity of $\mu$ we have
\[
   \begin{split}
     \mu_i'(E)
     &=\mu_{i+N}(R(E))\\
     &= \sum_{j \in \mathcal{I}_R} P_{j(i+N)} \mu_j (g_{i+N}^{-1}(R(E))) + \sum_{j \in \mathcal{I}_P} P_{(j+N)(i+N)} \mu_{j+N} (g_{i+N}^{-1}(R(E)))\\
     &= \sum_{j \in \mathcal{I}_R} P_{(j+N)i} \mu_{j+N}' (R(g_{i+N}^{-1}(R(E)))) + \sum_{j \in \mathcal{I}_P} P_{ji} \mu_j' (R(g_{i+N}^{-1}(R(E)))),
    \end{split}
\]
where for the latter we used  the symmetry of $\lambda$ and~\eqref{eq:defmirrored}.
Recalling now that for the fiber maps we have $g_i=R\circ g_{i+N}\circ R$, see~\eqref{eq:deffibermaps}, we obtain
\[
     \mu_i'(E)
     = \sum_{j \in \mathcal{I}_R} P_{(j+N)i} \mu_{j+N}' (g_i^{-1}(E)) + \sum_{j \in \mathcal{I}_P} P_{ji} \mu_{j}' (g_i^{-1}(E))
     = \sum_{j=1}^{2N} P_{ji} \mu_j'(g_i^{-1}(E)).
\]
Analogously, for  $i \in \ak N+1,\ldots,2N \fk$. This proves that $\mu'$ is stationary.

Let us see now that $\mu$ ergodic implies $\mu'$ also ergodic. Let $A \subset (\ak 1,\ldots,2N \fk \times I)$ satisfying $g_j^{-1}(A_i) = A_j$ for all $j,i$ such that $P_{ji} > 0$.
Note that
\[
    A':= \bigcup_{i=1}^{N} \{i\} \times (R(A_{i+N})) \cup \bigcup_{i=N+1}^{2N} \{i\} \times (R(A_{i-N}))
\]
satisfies the same property. In fact, it suffices to show that $g_j^{-1}(A_i') = A_j'$ for $i,j=1,\ldots,N$ such that $P_{ji} > 0$, since the other cases are analogous.
To see that this is true, note that the definitions of $A'$ and of the fiber maps of $G$ imply that
\[
    g_j^{-1}(A_i') = g_j^{-1}(R(A_{i+N})) = R(g_{j+N}^{-1}(A_{i+N})) = R(A_{j+N}) = A_j'
\]
Finally, to see that $\mu_i'(A_i) = 1$ for all $i=1,\ldots,2N$, note that if $i\in \ak 1,\ldots,N \fk$, the definitions of $\mu'$ and $A'$, the invariance of $A'$ and the ergodicity of $\mu$ imply that
\[
    \mu_i'(A_i) = \mu_{i+N}(R(A_i)) = \mu_{i+N}(A_{i+N}') = 1.
\]
As the case where $i\in \ak N+1,\ldots,2N \fk$ is analogous, we have that $\mu_i'(A_i) = 1$ for all $i=1,\ldots,2N$. As the invariant set $A$ was arbitrary, it follows that $\mu'$ is also ergodic. This concludes the proof.
\end{proof}


\subsection{Genericity conditions}\label{subsec:cg}

In this section we will verify the genericity conditions claimed in Theorem \ref{teo:TP}.

Fix $N \geq 1$. As above, we denote by $\mathcal{S}(N)$ the set of all step skew-products $F$ as in Definition \ref{defdeF}. Let $\mathcal{P}(N) \subset \mathcal{S}(N)$ be the subset of all those for which all fiber maps preserve orientation and $\mathcal{R}(N) := \mathcal{S}(N)\setminus \mathcal{P}(N)$. As $N$ is fixed, we will only write $\mathcal{S} := \mathcal{S}(N)$, $\mathcal{P}:=\mathcal{P}(N)$ and  $\mathcal{R}:=\mathcal{R}(N)$.
The sets $\mathcal{S}$, $\mathcal{P}$ and $\mathcal{R}$ are metric spaces when equipped with the distance $d$ defined by \begin{equation}\label{eq:distC1}
    d(F,H) := \max \ak d_{C^1} (f_i,h_i) \colon i=1,\ldots,N \fk.
\end{equation}
We state below the three conditions in $\mathcal{R}$ for Theorem \ref{teo:TP} be valid (recall that \cite[2.15 Theorem]{KV} already claims Theorem \ref{teo:TP} is true for the space $\mathcal{P}$).

\begin{enumerate}
\item[i)] \textbf{(Short periodic orbits are hyperbolic.)} Every fixed point of every composition $f_{\xi_1\ldots\xi_n}$, with $a_{\xi_n\xi_1} = 1$ and $1 \leq n \leq 2N$, is hyperbolic.
\item[ii)] \textbf{(Nonexistence of heteroclinic orbits.)} No attracting (repelling) fixed point of a map $f_{\xi_1\ldots\xi_n}$, with $a_{\xi_n\xi_1} = 1$ and $1 \leq n \leq 2N$, is sent to a repelling (an attracting) fixed point of a map $f_{\eta_1\ldots\eta_m}$ satisfying  $a_{\eta_m\eta_1} = 1$ and $1 \leq m \leq 2N$, by a composition $f_{\rho_1\ldots\rho_l}$, with $1 \leq l\leq 2N-1$.
\item [iii)] \textbf{(Nonexistence of cycles.)} There do not exist two points $a,b \in I$ such that
\[
     \left\{
        \begin{array}{ll}
          f_i(a) = a \quad \text{and} \quad R(f_i(R(b))) = b \quad \text{for all} \,\,i \in \mathcal{I}_P\\
          R(f_i(a)) = b \quad \text{and} \quad f_i(R(b)) = a \quad \text{for all} \,\,i \in \mathcal{I}_R.
        \end{array}
     \right.
\]
\end{enumerate}

The main aim of this section is to show the following result.

\begin{proposition}\label{prop:abeden}
	The subset $\widehat{\mathcal R}\subset\mathcal R$ of the step skew-products satisfying conditions i), ii) and iii) is open and dense with respect to the distance $d$ defined in~\eqref{eq:distC1}.
\end{proposition}

Note that the three conditions above are similar to the conditions imposed by Kleptsyn and Volk in \cite{KV}. By our choice of the extended step skew-product, the following result is then immediate and stated without explicit proof.

\begin{proposition}\label{prop:GsatisKV}
	For every $F\in\widehat{\mathcal R}$, the corresponding extended step skew-product $G$ satisfies the genericity  conditions of \cite[2.15 Theorem]{KV}.	
\end{proposition}

\begin{proof}[Proof of Proposition \ref{prop:abeden}]
\setcounter{claim}{0}
It suffices to show that any of the conditions i), ii), and iii)  hold in an open and dense subset of $\mathcal{R}$.

To show that conditions i) and iii) holds in open and dense sets $\widehat{\mathcal R}_i$ and $\widehat{\mathcal R}_{iii}$, respectively, follows standard arguments, to show i) see for example   \cite[Chapter XI, Sections 2 and 3 and in particular Parametric Transversality Theorem 2.3]{R}.

To show that condition $ii)$  holds in an open and dense subset $\widehat{\mathcal R}_{ii}\subset\mathcal{R}$, observe that $\widehat{\mathcal R}_{ii}$ can be written as
\[
	\widehat{\mathcal R}_{ii}
	= \bigcap_{\xi,\eta,\rho}\mathcal R_{\xi\eta\rho}
	\cap \bigcap_{\xi,\eta,\rho}\widehat{\mathcal R}_{\xi\eta\rho} ,
\]
where the first (the second) intersection is taken over all triples of finite sequences $\xi=\xi_1\ldots\xi_n$, $\eta=\eta_1\ldots\eta_m$, and $\rho=\rho_1\ldots\rho_l$ with $1 \leq n,m \leq 2N$ and $1 \leq l \leq 2N-1$ of all step skew-products such that no attracting (repelling) fixed point of $f_{\xi}$
is mapped to a repelling (an attracting) fixed point of $f_{\eta}$ by $f_{\rho}$.
Thus, it suffices to show that each  of these sets is open and dense.
The proof that theses sets are open is done by similar arguments as for condition i) and is omitted. Hence it suffices to prove the following claim (the corresponding claim for $\widehat{\mathcal R}_{\xi\eta\rho}$ is similar and also omitted).

\begin{claim}
The set $\mathcal{R}_{\xi\eta\rho}$ is dense in $\mathcal{R}$ for each triple of sequences $\xi, \eta,\rho$ as above.
\end{claim}

\begin{proof}
Fix $\xi,\eta$, and  $\rho$ and consider $F \in \widehat{\mathcal R}_i \setminus \mathcal{R}_{\xi\eta\rho}$. As $1 \leq l \leq 2N-1$, there exists a symbol that appears only once in $\rho=\rho_1\ldots\rho_l$. To simplify notation, suppose that this symbol is $\rho_l$ (the other cases are analogous).

The proof is done by small perturbations of  the map $f_{\rho_l}$ (in the $C^1$ topology), replacing it by a map $h$ while keeping the maps $f_i$, $i\ne \rho_l$, fixed and hence obtaining a step skew-product $F_h$ which is close to $F$ (with respect to $d$ defined in \eqref{eq:distC1}) and satisfies $F_h \in \mathcal{R}_{\xi\eta\rho}$.

If $\rho_l$ neither appears as a symbol in $\xi_1\ldots\xi_n$ nor in $\eta_1\ldots\eta_m$, the existence of such perturbations follows easily since in this case neither the attracting fixed points of $f_{{\xi_1}\ldots{\xi_n}}$ nor the repelling fixed points of $f_{{\eta_1}\ldots{\eta_m}}$ are affected.
The more general case is a bit more involved and will be treated below.

Let $\ak p_1, \ldots, p_r \fk$ be the set of attracting fixed points of $f_{\xi_1\ldots\xi_n}$ and $\ak q_1,\ldots,q_r\fk$ the set of repelling fixed points of $f_{\eta_1\ldots\eta_m}$ such that for each  $i=1,\ldots,r$ there exists $q_j$, $j\in\{1,\ldots,r\}$, with $f_{\rho_1\ldots\rho_l}(p_i) = q_j$ (the choice of $F$ implies that these sets are nonempty). Reordering the index, we can assume that $j=i$.
Consider the family $\{\widetilde f_i\}_{i=1,\ldots,N}$ defined by $\widetilde f_j=f_j$ if $j\ne\rho_l$ and $\widetilde f_j=h$ otherwise and denote by $\widetilde F=\widetilde F(h)$ the corresponding step skew-product.
Consider $\varepsilon> 0$ such that if $d_{C^1}(h,f_{\rho_l}) < \varepsilon$
then we have
\begin{itemize}
\item $\widetilde f_{\xi_n}\circ\ldots\circ\widetilde f_{\xi_1}$ has only  one attracting fixed point $p_i^h\in (p_i-\varepsilon,p_i+\varepsilon)$, $i=1,\ldots,r$, and coincides with $f_{\xi_1\ldots\xi_n}$ outside these intervals;
\item $\widetilde f_{\eta_m}\circ\ldots\circ\widetilde f_{\eta_1}$ has only one repelling fixed point $q_i^h\in(q_i - \varepsilon, q_i+\varepsilon)$, $i=1,\ldots,r$, and coincides with  $f_{\eta_1\ldots\eta_m}$  outside these intervals.
\end{itemize}
For each $i=1,\ldots,r$, define the map
\[
    \phi_i\colon B_{C^1(I,\real)}(f_{\rho_l},\varepsilon) \to \real,
     \quad  h \mapsto h \circ f_{\rho_1\ldots\rho_{l-1}} (p_i^h) - q_i^{h}.
\]
Given $h \in B_{C^1(I,\real)}(f_{\rho_l},\varepsilon)$, we have  $\widetilde F\in \mathcal{R}_{\xi\eta\rho}$
if $h \in \mathcal B_{C^1(I,\real)}(f_{\rho_l},\varepsilon) \setminus \bigcup_{i=1}^{r} \phi_i^{-1} (0)$.

Given $\delta > 0$ sufficiently small, we will consider a map (which will be further specified at the end of the proof)
\[
    \kappa\colon B_\delta(0) \to B_{C^1(I,\real)}(f_{\rho_l},\varepsilon),
     \quad t \mapsto h_t,
\]
so that $\kappa(0) = f_{\rho_l}$. Write
\[
	 f^{(t)}_{\xi_1\ldots\xi_n}
	= \widetilde f_{\xi_n}\circ\ldots\circ\widetilde f_{\xi_1}
	\quad \text{and}\quad
	 f^{(t)}_{\eta_1\ldots\eta_n}
	= \widetilde f_{\eta_m}\circ\ldots\circ\widetilde f_{\eta_1},
\]
where the maps $\widetilde f_j$ are defined as above considering the specific perturbation $h=h_t$ and denote by $F^{(t)}$ the new skew-product obtained. In this way we  consider, for each $i=1,\ldots,r$, the real function $\Phi_i := \phi_i \circ \kappa$.

Below we will choose  $\kappa$ in such a way that $\Phi_i'(0) \neq 0$ for all $i=1,\ldots,r$, then $\Phi_i^{-1}(\real\setminus{0})$ will be dense in $B_\delta(0)$. In particular, it will be possible to approach $f_{\rho_l}$ by a sequence of maps $h_t$ so that $F^{(t)} \in \mathcal{R}_{\xi\eta\rho}$, which will conclude the proof of the claim.

Assuming that $\kappa$ was already chosen,
let us calculate $\Phi_i'(0)$.
\[
\begin{split}
    &\Phi_i'(0)
    = \lim_{t \to 0} \frac1t(\Phi_i(t) - \Phi_i(0))
    = \lim_{t \to 0} \frac1t\Big(h_t (f_{\rho_1\ldots\rho_{l-1}}) (p_i^{h_t}) - q_i^{h_t} - \big(f_{\rho_1\ldots\rho_{l}} (p_i) - q_i\big)\Big)\\
    &=
    	\lim_{t \to 0} \frac{(h_t - f_{\rho_l}) (f_{\rho_1\ldots\rho_{l-1}} (p_i^{h_t})}{t}
		 + \lim_{t \to 0}\frac{f_{\rho_1\ldots\rho_{l}} (p_i^{h_t}) - f_{\rho_1\ldots\rho_{l}} (p_i)}{t}
		  - \lim_{t \to 0}\frac{q_i^{h_t} - q_i}{t}\\
	&=: A+B+C.		
\end{split}
\]
Note that
\[
	A
	= \frac{\partial h_t}{\partial t}(f_{\rho_1\ldots\rho_{l-1}}(p_i))
\]
and recall that we assume that the symbol $\rho_l$ does not appear in $\rho_1\ldots\rho_{l-1}$.
Note that the existence of the limit $C$ is nothing but the derivative of the hyperbolic continuation  (its derivative is given by the Implicit Function Theorem)
\begin{equation}\label{eq:ii.5}
    C
    = \frac{\partial q_i^{h_t}}{\partial t}|_{t=0}
    = \frac{1}{(f_{\eta_1...\eta_m})'(q_i)-1}
    		\frac{\partial f_{{\eta_1}...{\eta_m}}^{(t)}}{\partial t}(q_i).
\end{equation}
The existence of the limit $B$ is analogous.
\begin{equation}\label{eq:ii.4}
  B
  =  (f_{\rho_1...\rho_l})'(p_i) \frac{\partial p_i^{h_t}}{\partial t}|_{t=0}
  =  (f_{\rho_1...\rho_l})'(p_i)
  		\frac{1}{(f_{\xi_1...\xi_n})'(p_i)-1}
		\frac{\partial f_{{\xi_1}...{\xi_n}}^{(t)}}{\partial t}(p_i).
\end{equation}
Now we choose the one-parameter family $h_t$ (and hence our map $\kappa$) such that
\begin{itemize}
\item  the derivatives in \eqref{eq:ii.4} and \eqref{eq:ii.5} both are equal to $0$ (this is possible since each of them depends only on the value of the perturbation at $p_i$ and at $q_i$, respectively);
\item the limit $A$ is different from $0$ (this is possible since it is calculated at the hyperbolic continuation, which also changes and is different of $p$).
\end{itemize}
This conclude the proof of the claim.
\end{proof}


This finishes the proof of the proposition.
\end{proof}

\subsection{General structure of the (bi-)strips}\label{subsec:egf}

In this section we will assume that the hypotheses of Theorem \ref{teo:TP} are satisfied. We will analyze in detail the structure of the attracting and repelling bi-strips.

Let $F \in \widehat{\mathcal R}$ and $G$ be the corresponding extended step skew-product.  Let $\lambda_0$ be a nondegenerate Markov measure on $\Sigma_N$ and let $\lambda$ be its symmetric extension (recall Lemma \ref{lem:unisymext}).

\subsubsection{Topological structure}

In this section we will show that attracting (repelling) symmetric strips with respect to the extended step skew-product $G$ correspond to the same attracting (repelling) bi-strip with respect to $F$. We will show also that there exists a certain order among the bi-strips associated to $F$.

Recall that by Proposition \ref{prop:GsatisKV}, we can apply \cite[2.15 Theorem]{KV} to $G$ and $\lambda$. Hence there exists a finite collection of attracting and repelling strips (with respect to $G$) $S_1,...,S_n\subset\Sigma_A \times I$ and $R_1,...,R_{n-1}\subset\Sigma_A \times I$, respectively, such that their union is the whole space $\Sigma_A \times I$.
Note that, in fact, the strips do not depend on the choice of the (nondegenerate) Markov measure $\lambda$ but are only determined by the support of the corresponding stationary measure (which is not altered when changing weights).
 Moreover, possibly after some reordering, we have
\begin{equation}\label{eq:ordem}
    S_1 < R_1 < S_2 <... < S_{n-1} < R_{n-1} < S_n,
\end{equation}
where the notation $S < R$ means that $x<y$ whenever $(\omega,x) \in S$ and $(\omega,y) \in R$. Furthermore, the authors show that each attracting strip $S$ is of the form
\[
    S = \bigsqcup_{i=1}^{2N} [0;i] \times I_{\mu_i},
\]
where $I_{\mu_i}$ is a sufficiently small neighborhood of the closed interval which extremes are $\min \supp \mu_i$ and $\max \supp \mu_i$, where $\mu_i$ denotes the restriction $\mu|_{\{i\}\times I}$ of  an ergodic stationary measure $\mu$ (with respect to $G$).

By Lemma \ref{projdasfaixas}, for each attracting strip $S$, its projection $\Pi(S)$ is an attracting bi-strip (with respect to $F$). We can write
\[
    \Pi(S) = \Pi|_{C \times I}(S) \cup \Pi|_{(\Sigma_A\setminus C) \times I}(S),
\]
where $C$ is the set defined in \eqref{eq:defdeC}.
Note also that if $\mu$ is an ergodic stationary measure with respect to $G$, Lemma \ref{ergdemu'} implies that the mirrored measure of $\mu$, denoted by $\mu'$, is also ergodic and stationary. Denoting by $S$ the attracting strip obtained from $\mu$ and by $S'$ the attracting strip obtained from $\mu'$, the definition of $I_{\mu_i}$ implies that
\[
    \Pi|_{C \times I}(S) = \Pi|_{(\Sigma_A\setminus C) \times I}(S') \quad \text{and} \quad \Pi|_{(\Sigma_A \setminus C) \times I}(S) = \Pi|_{C \times I}(S')
\]
and then
\begin{equation}\label{eq:faixasimetrica}
    \Pi(S) = \Pi(S').
\end{equation}
The following result is also valid.

\begin{lemma}\label{lem:atmostone}
There exists at most one ergodic stationary measure $\mu$ with respect to $G$ such that $\mu' = \mu$.
\end{lemma}

\begin{proof}
By contradiction, suppose that there exist two ergodic stationary measures $\mu$ and $\nu$ such that $\mu' = \mu$ and $\nu' = \nu$. Then $I_{\mu'_1} = I_{\mu_1}$ and by the order of the attracting strips observed in \eqref{eq:ordem}, we should have either $I_{\nu_1} < I_{\mu_1}$ or $I_{\nu_1} > I_{\mu_1}$.
Suppose that $I_{\nu_1} < I_{\mu_1}$ (the other case is analogous). Using again the order of the attracting strips,
we should have
\[
    I_{\nu'_{N+1}} = I_{\nu_{N+1}} < I_{\mu_{N+1}} = I_{\mu'_{N+1}}.
\]
However, by the fact that $I_{\nu_1} < I_{\mu_1}$ and by the definitions of $\mu'$ and $\nu'$ we should also have that $I_{\mu'_{N+1}} < I_{\nu'_{N+1}}$, which gives a contradiction.  We are done.
\end{proof}

All the remarks above imply that if there exists an even number of attracting strips in $\Sigma_A \times I$, then they pairwise (one strip and its symmetric copy) project to the same image by $\Pi$. Otherwise, if there exists an odd number of attracting strips in $\Sigma_A \times I$, then all but one such strips pairwise (again, one strip and its symmetric copy) project to the same image by $\Pi$.
The same conclusion is valid for the repelling strips in $\Sigma_A \times I$. Furthermore, the images by $\Pi$ of the attracting and repelling strips do not intersect each other, since the projection $\Pi$ is 2-to-1 and the strips are all attracting/repelling strips are pairwise disjoint.  Thus, if $n$ is even, \eqref{eq:ordem} implies that
\begin{multline*}
    \Pi|_{C \times I} (S_1) < \Pi|_{C \times I}(R_1) < \Pi|_{C \times I}(S_2) < ...\\
    < \Pi|_{C \times I}(S_{\frac{n}{2}}) < \Pi|_{C \times I} (R_{\frac{n}{2}}) = \Pi|_{(\Sigma_A \setminus C) \times I} (R_{\frac{n}{2}}) < \Pi|_{(\Sigma_A \setminus C) \times I}(S_{\frac{n}{2}})< \\
    ...
    ... < \Pi|_{(\Sigma_A \setminus C) \times I} (R_1) < \Pi|_{(\Sigma_A \setminus C) \times I}(S_1).
\end{multline*}
If $n$ is odd then we have
\begin{multline*}
    \Pi|_{C \times I} (S_1) < \Pi|_{C \times I}(R_1) < \Pi|_{C \times I}(S_2) < ...\\
    < \Pi|_{C \times I}(R_{\frac{n}{2}}) < \Pi|_{C \times I} (S_{\frac{n+1}{2}})
    = \Pi|_{(\Sigma_A \setminus C) \times I} (S_{\frac{n+1}{2}})
    < \Pi|_{(\Sigma_A \setminus C) \times I}(R_{\frac{n}{2}})< \\...
    ... < \Pi|_{(\Sigma_A \setminus C) \times I} (R_1) < \Pi|_{(\Sigma_A \setminus C) \times I}(S_1).
\end{multline*}
In any case, we obtain the order for the bi-strips as claimed in Theorem \ref{teo:TP}.

The following lemma, which proof is analogous to the Lemma \ref{lem:simatrator}, relates the maximal attractors of symmetric strips.

\begin{lemma}
Let $S$ and $S'$ be two attracting (repelling) strips in $\Sigma_A \times I$ such that $\Pi(S)= \Pi(S')$. If $B$ and $B'$ are their maximal attractors, then $\Pi(B) = \Pi(B')$.
\end{lemma}


\subsubsection{Structure of the measures}

In this section we will relate ergodic invariant measures with respect to $F$ with invariant measures with respect to the extended step skew-product.

Given  $(\xi,z) \in \Sigma_N \times I$ and $(\omega,x) \in \Sigma_A \times I$,  for each $n \geq 1$ define
\[
    \nu_n(\xi,z) := \frac{1}{n} \sum_{i=0}^{n-1} \delta_{F^i(\xi,z)}
    \quad\text{ and }\quad
    \mu_n(\omega,x) := \frac{1}{n} \sum_{i=0}^{n-1} \delta_{G^i(\omega,x)}.
\]

The following is an immediate consequence of continuity.

\begin{lemma}\label{lem:levantamento}
If $\nu_n(\Pi(\omega,x)) \to \nu$ as $n\to\infty$ in the weak$\ast$ topology and $\mu$ is a weak$\ast$ limit point of the sequence $(\mu_n(\omega,x))$, then $\nu = \Pi_*\mu$.
\end{lemma}

By Lemma \ref{lem:levantamento}, it is an immediate consequence that given any ergodic $F$-invariant measure $\nu$, there exists a $G$-invariant measure $\mu$ such that $\nu = \Pi_*\mu$. However,  $\mu$ does not need to be unique nor ergodic, as in the following example.

\begin{example}
Let $f_1,f_2\colon I \to \ine(I)$ be two maps such that $f_1$ preserves  and $f_2$ reverses orientation. Suppose that $f_1$ has a fixed point $p_1$ and let $\nu := \delta_{((1)^\mathbb{Z}, p_1)}$, which is   $F$-invariant ergodic. Note that $\Pi{-1}(((1)^\mathbb{Z}, p_1))=\{((1)^\mathbb{Z}, p_1),((3)^\mathbb{Z}, 1-p_1)\}$ and that both measures $\mu := \delta_{((1)^\mathbb{Z}, p_1)}$ and $\widehat{\mu} := \delta_{((3)^\mathbb{Z}, 1-p_1)}$ are $G$-invariant ergodic and satisfy $\Pi_*\mu = \nu$ and $\Pi_*\widehat{\mu} = \nu$. Furthermore, for each $t \in (0,1)$, the measure $\mu_t := t\mu + (1-t) \widehat{\mu}$ is also $G$-invariant satisfies $\Pi_*\mu_t = \nu$, however  is not ergodic.
\end{example}

By \cite[2.15 Theorem]{KV}, given an attracting strip $S$ (with respect to $G$), there exists an unique ergodic $G$-invariant measure $\mu = \mu_S$ projecting to the Markov measure $\lambda$, the symmetric extension of $\lambda_0$.

\begin{lemma}\label{unicidadeaf3}
Let $S_0 = \Pi(S) \subset \Sigma_N \times I$ be an attracting bi-strip with respect to $F$. Then the unique ergodic $F$-invariant measure in $S_0$ which projects to $\lambda_0$ is the measure $\Pi_*\mu$, where $\mu = \mu_S$.
\end{lemma}


\begin{proof}
First, note that if $B_0$ is the maximal attractor of $S_0$ then
\begin{equation}\label{eq:suppnografico}
    \supp(\nu) \subset B_0 \quad \text{for all ergodic $F$-invariant measure $\nu$ in $S_0$}.
\end{equation}
Combining \eqref{eq:ordem}, \eqref{eq:faixasimetrica} and Lemma \ref{lem:atmostone}, we have that $S_0 = \Pi(S) = \Pi(S')$, where $S'$ is the mirrored strip of $S$ (the strip obtained from $\mu_{S'}$), and either $S = S'$ (which happens for at most one strip $S_0$) or $S \cap S' = \emptyset$. We analyze these cases separately.

\smallskip
\noindent\textbf{Case $S=S'$.} In this case $S_0$ is a (simple) strip and there exists a subset $D \subset \Sigma_N$ such that $\lambda_0(D) = 1$ and the projection $\Pi_1\colon B_0 \cap \Pi_1^{-1}(D) \to D$ is a bijection. This fact and \eqref{eq:suppnografico} imply that in this case $\Pi_*\mu$ is the unique measure in $S_0$ which projects to $\lambda_0$.

\smallskip
\noindent\textbf{Case $S\cap S' =\emptyset$.} In this case the bi-strip $S_0$ is not a (simple) strip and we can not apply the argument of the previous case. However, note that in this case, $\Pi|_S\colon S \to S_0$ and $\Pi|_{S'}\colon S' \to S_0$ are bijections such that $\Pi|_S\circ G=F\circ \Pi|_S$ and $\Pi|_{S'}\circ G=F\circ \Pi|_{S'}$. Suppose that, besides $\Pi_\ast\mu$, $\nu$ there is another measure in $S_0$ satisfying the conditions of the lemma. As $\Pi|_S\colon S \to S_0$ and $\Pi|_{S'}\colon S' \to S_0$ are bijections, we can consider the measures
\[
    \widehat{\mu} := ((\Pi|_S)^{-1})_*\nu \quad\text{and} \quad \widehat{\mu}' := ((\Pi|_{S'})^{-1})_*\nu.
\]
As $\nu$ is $F$-invariant and ergodic, we have that $\widehat{\mu}$ and $\widehat{\mu}'$ are $G$-invariant and ergodic.
\setcounter{claim}{0}
\begin{claim}
The measures $\widehat{\mu}$ and $\widehat{\mu}'$ are mirrored.
\end{claim}

\begin{proof}
Fix $i \in \ak 1,...,N \fk$. As the strips $S$ and $S'$ are mirrored, we have (recall \eqref{eq:cilindro})
\[
\begin{split}
    \widehat{\mu}([0;i] \times I)
    &= \widehat{\mu}(([0;i] \times I) \cap S)\\
   (\text{by definition of }\widehat\mu)\quad
    &= \nu(\Pi|_S(([0;i] \times I) \cap S)) \\
    (\text{mirror of the strips})\quad
    &= \nu(\Pi|_{S'}(([0;N+i] \times I) \cap S')) \\
   (\text{by definition of }\widehat\mu')\quad
    &= \widehat{\mu}'(([0;N+i] \times I) \cap S') \\
    &= \widehat{\mu}'([0;N+i] \times I)
\end{split}
\]
The argument for cylinders of arbitrary length is  analogous.  This proves the claim.
\end{proof}

Let $\widehat\lambda$ and $\widehat\lambda'$ be the projections of $\widehat{\mu}$ and $\widehat{\mu}'$, respectively, on $\Sigma_A$. Then $\widehat\lambda$ and $\widehat\lambda'$ are symmetric measures. In fact, if $B$ is the maximal attractor of $S$ then, by \eqref{eq:suppnografico}, $\supp(\widehat{\mu}) \subset B$. As $\Pi_1^A|_B\colon B \to \Pi_1^A(B)$ is a bijection (modulus a set of zero $\widehat\mu$-measure), it follows, from Claim 1, that $\widehat\lambda$ is symmetric. The proof of the symmetry of $\widehat\lambda'$ is analogous.
Let us see that
\[
    \pi_*\widehat{\lambda} = \lambda_0.
\]
We will show again this equality for cylinders with length $1$ and the general argument is analogous. Fix $i \in \ak 1,...,N\fk$.
Then
\[
\begin{split}
    \widehat{\lambda}(\pi^{-1}([0;i]))
    &= \widehat{\lambda}([0;i] \cup [0;N+i])
    = \widehat{\lambda}([0;i]) + \widehat{\lambda}([0;N+i])\\
    &= \widehat{\mu}([0;i] \times I) + \widehat{\mu}([0;N+i] \times I)
    =   \widehat{\mu}([0;i] \times I\cup[0;N+i] \times I)\\
    &= \widehat{\mu}(\Pi^{-1}([0;i] \times I))\\
    &= \nu([0;i] \times I)\\
    &= \lambda_0([0;i]).
\end{split}
\]
Analogously, we can show that $\widehat{\lambda}'$ projects to $\lambda_0$ and that $\widehat\lambda$ and $\widehat\lambda'$ are Markov measures in $\Sigma_A$. As both measures project to $\lambda_0$ and are symmetric, it follows that $\widehat\lambda=\widehat\lambda'$ and therefore $\widehat\lambda=\lambda$ (symmetric extension of $\lambda_0$ to $\Sigma_A$).

By  \cite[2.15 Theorem]{KV},  $\mu$  is the unique ergodic $G$-invariant measure projecting to $\lambda$. Hence we have $\widehat{\mu} = \mu$. Thus, $\nu = (\Pi|_S)_*\mu$.
\end{proof}


\subsection{Conclusion of the proof of Theorem \ref{teo:TP}}\label{subsec:demTP}

We are now prepared to conclude the proof of Theorem \ref{teo:TP}.
As \cite{KV} already implies the result for the space $\mathcal{P}$, it suffices to consider the space $\mathcal{R}$.

By Proposition~\ref{prop:abeden}, there exists an open and dense subset $\widehat{\mathcal R}\subset\mathcal R$ of step skew-products satisfying all the conditions i), ii) and iii) stated in Section \ref{subsec:cg}. Fix $F\in\widehat{\mathcal R}$. By Proposition \ref{prop:GsatisKV},we can apply \cite[2.15 Theorem]{KV} to  the extended step skew-product $G$. Hence, there exists a finite collection of attracting and repelling strips with respect to $G$ such that their union is the whole phase space $\Sigma_A \times I$. By Lemma \ref{projdasfaixas}, each attracting (repelling) strip with respect to $G$ is sent by $\Pi$ in an attracting (repelling) bi-strip with respect to $F$. As $\Pi$ is surjective, item 1 of the Theorem \ref{teo:TP} is proved.

Consider $S_0 = \Pi(S)$ an attracting bi-strip with respect to $F$, where $S$ is an attracting strip with respect to $G$, and denote by $B$ the maximal attractor of $S$. By \cite[$2.15$ Theorem]{KV}, $B$ is a continuous bony graph with respect to the Markov measure $\lambda$, the symmetric extension of $\lambda_0$. Then, by Lemmas \ref{lem:relentreatratores} and \ref{lem:duplobg},  the set $\Pi(B)$ is the maximal attractor of $S$ and is a continuous bi-bony graph.
Analogously, the maximal repeller of each repelling bi-strip is a continuous bi-bony graph.
This proves the item 2 of  Theorem \ref{teo:TP}.

In order to prove the item 3 of  Theorem \ref{teo:TP}, consider again $S_0 = \Pi(S)$ an attracting bi-strip with respect to $F$, where $S$ is an attracting strip with respect to $G$, and denote by $B_0$ its maximal attractor. By \cite[2.15 Theorem]{KV}, there exists an unique invariant and ergodic measure (with respect to $G$) $\mu$ in $S$ such that $(\Pi_1^A)_* \mu = \lambda$, the symmetric extension of $\lambda_0$. Moreover, $\mu$ is physical and hyperbolic and its basin contains a full measure set in $S$. By Lemma \ref{lem:af3teo}, the measure $\Pi_* \mu$ is $F$-invariant and ergodic and projects to $\lambda_0$. Furthermore, by Lemma \ref{unicidadeaf3}, it is the unique measure with this property. By Lemma \ref{lem:medfisica} and Corollary \ref{lem:medhiperbolica}, $\Pi_\ast\mu$ is physical, hyperbolic and its basin contains a full measure set in the bi-strip $S_0$. This concludes the proof of the item 3 of Theorem \ref{teo:TP} in the case of attracting bi-strips. The case of repelling bi-strips is analogous.

Finally, to prove item 4, consider one more time $S_0 = \Pi(S)$ an attracting bi-strip with respect to $F$ and let $\nu$ be the measure obtained in item 3. Suppose that $S_0$ is a (simple) strip. By item 2 proved above and by \eqref{eq:suppnografico}, there exists a subset $\Upsilon \subset B_0$ such that $\nu(\Upsilon) = 1$ and the projection $\Pi_1|_\Upsilon\colon\Upsilon \to \Pi_1(\Upsilon)$ is a bijection. As it is immediate that
\[
    \Pi_1 \circ F = \sigma \circ \Pi_1,
\]
we have that $\Pi_1|_\Upsilon$ is a conjugation between $\Upsilon$ and $\Pi_1(\Upsilon)$. The proof in the case where $S_0$ is a bi-strip is analogous, but $\Pi_1|_\Upsilon$ becomes a surjective map two-to-one instead of a bijection. This proves item 4 in the case of attracting strips.
The case of repelling strips is analogous.

Therefore, Theorem \ref{teo:TP} is proved.
\qed

\bibliographystyle{plain}

\end{document}